\documentclass[12 pt]{amsart}
\usepackage[pdftex]{graphicx}
\usepackage{amsmath}
\usepackage{mathtools}
\usepackage{amsthm}
\usepackage{amsfonts}
\usepackage{amssymb}
\usepackage[margin=1.0in]{geometry}
\usepackage{tikz}
\usepackage{enumitem}
\usepackage{url}

\newtheorem{theorem}{Theorem}[section]
\newtheorem*{theorem*}{Theorem}

\newtheorem{lemma}[theorem]{Lemma}
\newtheorem{prop}[theorem]{Proposition}
\newtheorem{que}[theorem]{Question}

\newtheorem{conj}[theorem]{Conjecture}

\theoremstyle{definition}
\newtheorem{defin}[theorem]{Definition}
\newtheorem{fact}[theorem]{Fact}

\theoremstyle{remark}
\newtheorem*{rem}{Remark}

\newcommand{\flim}[1]{\mathrm{Flim}(#1)}
\newcommand{\age}[1]{\mathrm{Age}(#1)}
\newcommand{\fin}[1]{\mathrm{Fin}(#1)}
\newcommand{\fr}{Fra\"iss\'e }
\renewcommand{\phi}{\varphi}

\newcommand{\emb}[1]{\mathrm{Emb}(#1)}
\newcommand{\aut}[1]{\mathrm{Aut}(#1)}
\newcommand{\homeo}[1]{\mathrm{Homeo}(#1)}

\newcommand{\ran}[1]{\mathrm{ran}(#1)}

\newcommand{\cupdots}{\cup\cdots\cup}
\newcommand{\hatf}{\,\hat{\rule{-0.5ex}{1.5ex}\smash{f}}}

\newcommand{\tildei}[2]{\tilde{\imath}_{#1}^{#2}}

\begin{document}
\title[Fra\"iss\'e structures and a conjecture of Furstenberg]{Fra\"iss\'e structures and a conjecture of Furstenberg}
\author[D. Barto\v{s}ov\'a and A. Zucker]{Dana Barto\v{s}ov\'a and Andy Zucker}
\maketitle

\begin{abstract}
	We study problems concerning the Samuel compactification of the automorphism group of a countable first-order structure. A key motivating question is a problem of Furstenberg and a counter-conjecture by Pestov regarding the difference between $S(G)$, the Samuel compactification, and $E(M(G))$, the enveloping semigroup of the universal minimal flow. We resolve Furstenberg's problem for several automorphism groups and give a detailed study in the case of $G= S_\infty$, leading us to define and investigate several new types of ultrafilter on a countable set.
\end{abstract}

\section{Introduction}

\let\thefootnote\relax\footnote{A.\ Zucker was supported by NSF Grants no.\ DGE 1252522 and DMS 1803489.}
In this paper, we are interested in the automorphism groups of countable first-order structures and the Samuel compactifications of these groups. We will address a variety of questions about the algebraic structure of the Samuel compactification and exhibit connections between this algebraic structure and the combinatorics of the first-order structures at hand.

Let $G$ be a topological group; all topological groups and spaces will be assumed Hausdorff. The group $G$ comes with a natural uniform structure, the \emph{left uniformity}, whose entourages are of the form $\{(g,h)\in G\times G: g^{-1}h\in V\}$ where $V$ ranges over open symmetric neighborhoods of the identity. Every uniform space $U$ admits a \emph{Samuel compactification}, the Gelfand space of the algebra of bounded uniformly continuous functions on $U$ (see [Sa] or [U]). We denote by $S(G)$ the Samuel compactification of the group $G$ with its left uniform structure.

In addition to being a compact Hausdorff space, the space $S(G)$ can also be endowed with a $G$-flow structure. A \emph{$G$-flow} is a compact Hausdorff space $X$ equipped with a continuous right $G$-action $a: X\times G\rightarrow X$. Typically the action $a$ is understood, and we write $x\cdot g$ or $xg$ for $a(x,g)$. We can give $S(G)$ the structure of a $G$-flow; indeed, for each $g\in G$, the right-multiplication map $h\rightarrow hg$ is left-uniformly continuous, so can be continuously extended to $S(G)$. With some extra work, it can be shown that the evaluation $S(G)\times G\rightarrow S(G)$ is continuous. 

If $X$ and $Y$ are $G$-flows, a \emph{$G$-map} is a continuous map $\phi: X\rightarrow Y$ which respects the $G$-action. A \emph{$G$-ambit} is a pair $(X, x_0)$, where $X$ is a $G$-flow and $x_0\in X$ has a dense orbit. If $(X, x_0)$ and $(Y, y_0)$ are ambits, then a \emph{map of ambits} is a $G$-map $\phi: X\rightarrow Y$ with $\phi(x_0) = y_0$. Notice that there is at most one map of ambits from $(X, x_0)$ to $(Y, y_0)$. By identifying $G$ as embedded into $S(G)$ and by considering the orbit of $1_G$, we turn $(S(G), 1_G)$ into an ambit. It turns out that this is the \emph{greatest ambit}; for any $G$-ambit $(X, x_0)$, there is a map of ambits $\phi: (S(G), 1_G)\rightarrow (X,x_0)$.

We can use this universal property to endow $S(G)$ with yet more structure. A \emph{compact left-topological semigroup} is a semigroup $S$ with a compact Hausdorff topology in which the left multiplication maps $t\rightarrow st$ are continuous for each $s\in S$. Now let $x\in S(G)$; then the pair $(\overline{x\cdot G}, x)$ is a $G$-ambit, so there is a unique $G$-map $\lambda_x: S(G)\rightarrow \overline{x\cdot G}$ with $\lambda_x(1_G) = x$. We can endow $S(G)$ with the structure of a compact left-topological semigroup by setting $xy:= \lambda_x(y)$. It is not hard to show that this operation is associative.

Another consequence of the universal property is the existence of universal minimal flows. Let $X$ be a $G$-flow. A \emph{subflow} is any closed non-empty $Y\subseteq X$ which is invariant under the $G$-action. The $G$-flow $X$ is \emph{minimal} if $X\neq\emptyset$ and every orbit is dense; equivalently, $X$ is minimal if $X$ contains no proper subflows. An easy Zorn's Lemma argument shows that every flow contains a minimal subflow.  The $G$-flow $X$ is \emph{universal} if there is a $G$-map from $X$ onto any minimal flow. Let $M\subseteq S(G)$ be any minimal subflow, and let $Y$ be a minimal flow. Pick $y_0\in Y$ arbitrarily, making $(Y, y_0)$ an ambit. Then there is a map of ambits $\phi: (S(G), 1_G)\rightarrow (Y, y_0)$, and $\phi|_M: M\rightarrow Y$ is a $G$-map. We have just shown the existence of a \emph{universal minimal flow}. By using some techniques from the theory of compact left-topological semigroups, it can be shown that there is a unique universal minimal flow up to $G$-flow isomorphism, denoted $M(G)$.

The existence and uniqueness of the universal minimal flow suggests another ``canonical'' $G$-ambit we can construct. If $X$ is a $G$-flow, we can view each $g\in G$ as the function $\rho_g: X\rightarrow X$. Form the product space $X^X$, and set $E(X) = \overline{\{\rho_g: g\in G\}}$. It will be useful in this instance to write functions on the right, so if $f\in X^X$, we write $x\cdot f$ or $xf$ instead of $f(x)$. The group $G$ acts on $E(X)$ via $x\cdot(f\cdot g) = (x\cdot f)\cdot g$. Notice that $\rho_g\cdot h = \rho_{gh}$, so $(E(X), \rho_{1_G})$ is an ambit. We can also give $E(X)$ a compact left-topological semigroup structure; rather than a universal property, it is the fact that members of $E(X)$ are functions that allows us to do this. Indeed, the product is given by composition, which with our notation means that for $f_1, f_2\in E(X)$, we define $x\cdot(f_1\cdot f_2) = (x\cdot f_1)\cdot f_2$. The ambit $E(X)$ (the distinguished point being understood) is called the \emph{enveloping semigroup} of $X$; we will be particularly interested in $E(M(G))$, the enveloping semigroup of the universal minimal flow.

It is worth pointing out that we could have avoided some of this notational awkwardness by switching the roles of left and right throughout, i.e.\ working with \emph{left} $G$-actions and compact \emph{right}-topological semigroups. The reason we work with our left-right conventions is due to the specific groups that we will be working with, i.e.\ automorphism groups of countable first-order structures. We will point out later how a left-right switch could be made. Also note that several of the references use the opposite left-right conventions, in particular [HS] and [Ba]. 

Robert Ellis (see [E]) first proved the existence and uniqueness of $M(G)$, and was the first to construct the enveloping semigroup of a flow. Upon considering the two canonical ambits $S(G)$ and $E(M(G))$, we see that there is a map of ambits $\phi: S(G)\rightarrow E(M(G))$ (when referring to $S(G)$ and enveloping semigroups, we will suppress the distinguished point unless there is possible confusion). Historically, the following very natural question had been attributed to Ellis: is $\phi: S(G)\rightarrow E(M(G))$ an isomorphism? Vladimir Pestov (see [P]) observed that the existence of non-trivial \emph{extremely amenable} groups, groups where $M(G)$ is a singleton, provides a negative answer to Ellis's question. Pestov also constructed many other examples of groups $G$ where $S(G)$ and $E(M(G))$ were not isomorphic. The diversity of counterexamples to Ellis's question led Pestov to make the following conjecture.

\begin{conj}[\cite{P}, p.\ 4163]
\label{PestovCon}
Let $G$ be a topological group. Then the canonical map $\phi: S(G)\rightarrow E(M(G))$ is an isomorphism iff $G$ is precompact. 
\end{conj}
\vspace{2 mm}

Here, $G$ is said to be \emph{precompact} if the completion of its left uniformity is compact. If this is the case, then all of $S(G)$,  $M(G)$, and $E(M(G))$ are isomorphic to the left completion. Aside from the initial work of Pestov, most work done on Conjecture \ref{PestovCon} has been directed towards discrete groups. Glasner and Weiss in \cite{GW1} show that $S(\mathbb{Z})$ and $E(M(\mathbb{Z}))$ are not isomorphic. Much more recently, Glasner and Weiss in \cite{GW2} isolate a class of countable discrete groups they call \emph{DJ groups}. They verify Conjecture \ref{PestovCon} for every DJ group and show that many groups are DJ groups, including amenable groups and residually finite groups. They also assert that the attribution of the original question to Ellis is mistaken, and that instead Furstenberg conjectured a statement equivalent to Conjecture \ref{PestovCon} for $\mathbb{Z}$ in \cite{F} (see part III of \cite{F}). Hence our reference to Furstenberg in the title. It is unknown whether any countable discrete group is not DJ. 

In this paper, we address Conjecture \ref{PestovCon} for groups of the form $G = \aut{\mathbf{K}}$ where $\mathbf{K}$ is a countable first-order structure. We endow $G$ with the topology of pointwise convergence, turning $G$ into a Polish group. In a mild abuse of terminology, we will call groups of this form \emph{automorphism groups}. When $\mathbf{K}$ is a countable set with no additional structure, we have $\aut{\mathbf{K}} = S_\infty$, the group of all permutations of a countable set. More generally, automorphism groups are exactly the closed subgroups of $S_\infty$. The work of Kechris, Pestov, and Todorcevic [KPT] provides explicit computations of $M(G)$ for many automorphism groups. Having an explicit representation of $M(G)$ aids in analyzing the properties of $E(M(G))$. Along with an explicit representation of $S(G)$ for automorphism groups (see [Z]), this allows us to address Conjecture \ref{PestovCon} for some of these groups. Our first main theorem is the following.

\begin{theorem}
\label{PestovSolution}
Let $\mathbf{K}$ be any of the following:
\begin{itemize}
\item
a countable set without structure,
\item
the random $K_n$-free graph,
\item
the random $r$-uniform hypergraph.
\end{itemize}
Then for $G = \aut{\mathbf{K}}$, we have $S(G)\not\cong E(M(G))$.
\end{theorem}

It is interesting to note that the methods here and the methods from \cite{GW2} are orthogonal in some sense; our methods only work if $M(G)$ is metrizable, and topological groups with $M(G)$ metrizable never have property DJ (see Remark 1.2 from \cite{GW2}).

We then turn to finding the extent to which $S(G)$ and $E(M(G))$ differ. Any minimal subflow $M\subseteq S(G)$ is isomorphic to $M(G)$, and it turns out that $S(G)$ admits a \emph{retraction} onto $M$, i.e.\ a $G$-map $\phi: S(G)\rightarrow M$ with $\phi|_M$ the identity. Pestov has shown (see [P1]) that $S(G)\cong E(M(G))$ iff the retractions of $S(G)$ onto a minimal subflow $M\subseteq S(G)$ separate the points of $S(G)$. So if $S(G)\not\cong E(M(G))$, it makes sense to ask which pairs of points cannot be separated; this will not depend on the choice of minimal subflow $M\subseteq S(G)$. Given $x,y\in S(G)$, we say they can be \emph{separated by retractions} if there is a retraction $\phi: S(G)\rightarrow M$ with $\phi(x)\neq\phi(y)$. 

Every compact left-topological semigroup $S$ admits a smallest two-sided ideal, denoted $K(S)$. Our second main theorem is the following.
\vspace{3 mm}

\begin{theorem}
\label{PointsSmallestIdeal}
There are $x\neq y\in K(S(S_\infty))$ which cannot be separated by retractions.
\end{theorem}
\vspace{3 mm}

On the way to proving Theorem \ref{PointsSmallestIdeal}, we prove some theorems of independent interest both for general topological groups $G$ and for $S_\infty$. By a well-known theorem of Ellis, every compact left-topological semigroup $S$ contains an \emph{idempotent}, an element $u\in S$ which satisfies $u\cdot u = u$ (see [E]). Given $Y\subseteq S$, write $J(Y)$ for the set of idempotents in $Y$. Our route to proving Theorem \ref{PointsSmallestIdeal} involves a careful understanding of when the product of two idempotents is or is not an idempotent. 

In the case $G = S_\infty$, we are able to find large semigroups of idempotents; this is what allows us to prove Theorem \ref{PointsSmallestIdeal}.
\vspace{3 mm}

\begin{theorem}
\label{TwoIdeals}
There are two minimal subflows $M\neq N\subseteq S(S_\infty)$ so that $J(M)\cup J(N)$ is a semigroup.
\end{theorem}
\vspace{3 mm}

It is worth noting that any minimal subflow $M\subseteq S(G)$ is a compact subsemigroup of $S(G)$, so $J(M)\neq\emptyset$.

There are some cases when it is clear that $K(S(G))$ contains sufficiently large semigroups of idempotents. Given a $G$-flow $X$, recall that a pair of points $x,y\in X$ is called \emph{proximal} if there is $p\in E(X)$ with $xp = yp$; the pair $(x,y)$ is called \emph{distal} if it is not proximal. A $G$-flow $X$ is \emph{proximal} if every pair from $X$ is proximal, and $X$ is called \emph{distal} if every pair $x\neq y\in X$ is distal. If $M(G)$ is proximal, then whenever $M\subseteq S(G)$ is a minimal subflow, we have $J(M) = M$. If $M(G)$ is distal and $M\subseteq S(G)$ is a minimal subflow, then $J(M) = \{u\}$, a single idempotent. So long as $S(G)$ contains at least two minimal right ideals, which is always the case when $G$ is Polish (see [Ba]), then $E(M(G))\not\cong S(G)$ in these cases. 

The paper is organized as follows. Section 2 provides background on \fr structures, their automorphism groups, and the Samuel compactifications of these groups. Section 3 gives a review of KPT correspondence. Section 4 gives the proof of Theorem \ref{PestovSolution}, and section 5 gives the proofs of Theorems \ref{PointsSmallestIdeal} and \ref{TwoIdeals}. Section 6 gives a brief discussion of the case where $M(G)$ is proximal or distal. The last section, section 7, investigates some of the combinatorial content of section 5 and introduces some new types of ultrafilters on $[\omega]^2$. 
\vspace{5 mm}

\section{Countable first-order structures and the Samuel compactification}
\vspace{5 mm}

In this section, we provide the necessary background on countable structures and provide an explicit construction of the Samuel compactification of an automorphism group. The presentation here is largely taken from [Z1].

Recall that $S_\infty$ is the group of all permutations of $\omega := \{0,1,2,...\}$. We can endow $S_\infty$ with the topology of pointwise convergence; a typical basic open neighborhood of the identity is $\{g\in S_\infty: g(k) = k \text{ for every } k< n\}$ for some $n< \omega$. Notice that each of these basic open neighborhoods is in fact a clopen subgroup.

Fix now $G$ a closed subgroup of $S_\infty$. A convenient way to describe the $G$-orbits of finite tuples from $\omega$ is given by the notions of a \fr class and structure. A \emph{relational language} $L = \{R_i: i\in I\}$ is a collection of relation symbols. Each relation symbol $R_i$ has an \emph{arity} $n_i\in \mathbb{N}$. An \emph{$L$-structure} $\mathbf{A} = \langle A, R_i^\mathbf{A}\rangle$ consists of a set $A$ and relations $R_i^\mathbf{A}\subseteq A^{n_i}$; we say that $\mathbf{A}$ is an $L$-structure on $A$. If $\mathbf{A},\mathbf{B}$ are $L$-structures, then $g: \mathbf{A}\rightarrow \mathbf{B}$ is an \emph{embedding} if $g$ is a map from $A$ to $B$ such that $R_i^\mathbf{A}(x_1,...,x_{n_i})\Leftrightarrow R_i^\mathbf{B}(g(x_1),...,g(x_{n_i}))$ for all relations. We write $\emb{\mathbf{A}, \mathbf{B}}$ for the set of embeddings from $\mathbf{A}$ to $\mathbf{B}$. We say that $\mathbf{B}$ \emph{embeds} $\mathbf{A}$ and write $\mathbf{A}\leq \mathbf{B}$ if $\emb{\mathbf{A}, \mathbf{B}}\neq\emptyset$. An \emph{isomorphism} is a bijective embedding, and an \emph{automorphism} is an isomorphism between a structure and itself. If $A\subseteq B$, then we say that $\mathbf{A}$ is a \emph{substructure} of $\mathbf{B}$, written $\mathbf{A}\subseteq \mathbf{B}$, if the inclusion map is an embedding. $\mathbf{A}$ is finite, countable, etc.\ if $A$ is.
\vspace{2 mm}

\begin{defin}
Let $L$ be a relational language. A \emph{\fr class} $\mathcal{K}$ is a class of $L$-structures with the following four properties.
\vspace{2 mm}

\begin{enumerate}
\item
$\mathcal{K}$ contains only finite structures, contains structures of arbitrarily large finite cardinality, and is closed under isomorphism.
\vspace{2 mm}

\item
$\mathcal{K}$ has the \emph{Hereditary Property} (HP): if $\mathbf{B}\in \mathcal{K}$ and $\mathbf{A}\subseteq \mathbf{B}$, then $\mathbf{A}\in \mathcal{K}$.
\vspace{2 mm}

\item
$\mathcal{K}$ has the \emph{Joint Embedding Property} (JEP): if $\mathbf{A}, \mathbf{B}\in \mathcal{K}$, then there is $\mathbf{C}$ which embeds both $\mathbf{A}$ and $\mathbf{B}$.
\vspace{2 mm}

\item
$\mathcal{K}$ has the \emph{Amalgamation Property} (AP): if $\mathbf{A}, \mathbf{B}, \mathbf{C}\in \mathcal{K}$ and $f: \mathbf{A}\rightarrow \mathbf{B}$ and $g: \mathbf{A}\rightarrow \mathbf{C}$ are embeddings, there is $\mathbf{D}\in \mathcal{K}$ and embeddings $r: \mathbf{B}\rightarrow \mathbf{D}$ and $s:\mathbf{C}\rightarrow\mathbf{D}$ with $r\circ f = s\circ g$.
\end{enumerate}
\end{defin}
\vspace{3 mm}

If $\mathbf{K}$ is a countably infinite $L$-structure (which we will typically assume has underlying set $\omega$), we write $\age{\mathbf{K}}$ for the class of finite $L$-structures which embed into $\mathbf{K}$. The following is the major fact about \fr classes.
\vspace{2 mm}

\begin{fact}
If $\mathcal{K}$ is a \fr class, there is up to isomorphism a unique countably infinite $L$-structure $\mathbf{K}$ with $\age{\mathbf{K}} = \mathcal{K}$ satisfying one of the following two equivalent conditions.
\vspace{2 mm}

\begin{enumerate}
\item
$\mathbf{K}$ is \emph{ultrahomogeneous}: if $f: \mathbf{A}\rightarrow \mathbf{B}$ is an isomorphism between finite substructures of $\mathbf{K}$, then there is an automorphism of $\mathbf{K}$ extending $f$.
\vspace{2 mm}

\item
$\mathbf{K}$ satisfies the \emph{Extension Property}: if $\mathbf{B}\in \mathcal{K}$, $\mathbf{A}\subseteq \mathbf{B}$, and $f: \mathbf{A}\rightarrow \mathbf{K}$ is an embedding, there is an embedding $h: \mathbf{B}\rightarrow \mathbf{K}$ extending $f$.
\end{enumerate}
\vspace{2 mm}

Conversely, if $\mathbf{K}$ is a countably infinite $L$-structure satisfying 1 or 2, then $\age{\mathbf{K}}$ is a \fr class.
\end{fact}
\vspace{3 mm}

Given a \fr class $\mathcal{K}$, we write $\flim{\mathcal{K}}$, the \emph{\fr limit} of $\mathcal{K}$, for the unique structure $\mathbf{K}$ as above. We say that $\mathbf{K}$ is a \emph{\fr structure} if $\mathbf{K}\cong \flim{\mathcal{K}}$ for some \fr class. Our interest in \fr structures stems from the following result.
\vspace{3 mm}

\begin{fact}
For any \fr structure $\mathbf{K}$, $\aut{\mathbf{K}}$ is isomorphic to a closed subgroup of $S_\infty$. Conversely, any closed subgroup of $S_\infty$ is isomorphic to $\aut{\mathbf{K}}$ for some \fr structure $\mathbf{K}$.
\end{fact}
\vspace{3 mm}

Fix a \fr class $\mathcal{K}$ with \fr limit $\mathbf{K}$. Set $G = \aut{\mathbf{K}}$. We also fix an \emph{exhaustion} $\mathbf{K} = \bigcup_n \mathbf{A}_n$, $n\geq 1$, with each $\mathbf{A}_n\in \mathcal{K}$, $|\mathbf{A}_n| = n$, and $\mathbf{A}_m\subseteq \mathbf{A}_n$ for $m\leq n$. Whenever we write $\mathbf{K} = \bigcup_n \mathbf{A}_n$, it will be assumed that the right side is an exhaustion of $\mathbf{K}$. Write $H_n = \{gG_n: g\in G\}$, where $G_n = G\cap N_{\mathbf{A}_n}$ is the pointwise stabilizer of $\mathbf{A}_n$. We can identify $H_n$ with $\emb{\mathbf{A}_n, \mathbf{K}}$, the set of embeddings of $\mathbf{A}_n$ into $\mathbf{K}$. Note that under this identification, we have $H_n = \bigcup_{N\geq n} \emb{\mathbf{A}_n, \mathbf{A}_N}$. For $g\in G$, we often write $g|_n$ for $gG_n$, and we write $i_n$ for $G_n$. The group $G$ acts on $H_n$ on the left; if $x\in H_n$ and $g\in G$, we have $g\cdot x = g\circ x$. For $m\leq n$, we let $i^n_m\in \emb{\mathbf{A}_m, \mathbf{A}_n}$ be the inclusion embedding.
\vspace{3 mm}

Each $f\in \emb{\mathbf{A}_m, \mathbf{A}_n}$ gives rise to a dual map $\hatf: H_n\rightarrow H_m$ given by $\hatf(x) = x\circ f$. Note that we must specify the range of $f$ for the dual map to make sense, but this will usually be clear from context. 
\vspace{3 mm}

\begin{prop}\mbox{}
\label{Amalgamation}
\begin{enumerate}
\item
For $f\in \emb{\mathbf{A}_m, \mathbf{A}_n}$, the dual map $\hatf: H_n\rightarrow H_m$ is surjective.
\vspace{2 mm}
\item
For every $f\in \emb{\mathbf{A}_m, \mathbf{A}_n}$, there is $N\geq n$ and $h\in \emb{\mathbf{A}_n, \mathbf{A}_N}$ with $h\circ f = i^N_m$.
\end{enumerate}
\end{prop}

\begin{proof}
Item $(1)$ is an immediate consequence of the extension property. For item $(2)$, use ultrahomogeneity to find $g\in G$ with $g\circ f = i_m$. Let $N\geq n$ be large enough so that $\ran{g|_n}\subseteq \mathbf{A}_N$, and set $h = g|_n$.
\end{proof}

We now proceed with an explicit construction of $S(G)$. First, if $X$ is a discrete space, we let $\beta X$ be the space of ultrafilters on $X$. We topologize $\beta X$ by declaring a typical basic open neighborhood to be of the form $\{p\in \beta X: A\in p\}$, where $A\subseteq X$. We view $X$ as a subset of $\beta X$ by identifying $x\in X$ with the ultrafilter $\{A\subseteq X: x\in A\}$. If $Y$ is a compact Hausdorff space and $\phi: X\rightarrow Y$ is any map, there is a unique continuous extension $\tilde{\phi}: \beta X\rightarrow Y$.

Now let $f\in \emb{\mathbf{A}_m, \mathbf{A}_n}$. The dual map $\hatf$ extends to a continuous map $\tilde{f}: \beta H_n\rightarrow \beta H_m$. If $p\in \beta H_n$ and $f\in \emb{\mathbf{A}_m, \mathbf{A}_n}$, we will sometimes write $p\cdot f$ for $\tilde{f}(p)$. Form the inverse limit $\varprojlim \beta H_n$ along the maps $\tildei{m}{n}$. We can identify $G$ with a dense subspace of $\varprojlim \beta H_n$ by associating to each $g\in G$ the sequence of ultrafilters principal on $g|_n$. The space $\varprojlim \beta H_n$ turns out to be the Samuel compactification $S(G)$ (see Corollary 3.3 in [P]).

To see that $S(G)$ is the greatest ambit, we need to exhibit a right $G$-action on $S(G)$. This might seem unnatural at first; after all, the left $G$-action on each $H_n$ extends to a left $G$-action on $\beta H_n$, giving us a left $G$-action on $S(G)$. The problem is that the left action is not continuous when $G$ is given its Polish topology. The right action we describe doesn't ``live'' on any one level of the inverse limit $\varprojlim \beta H_n$; we need to understand how the various levels interact. 

Let $\pi_n: \varprojlim \beta H_n\rightarrow \beta H_n$ be the projection map. We often write $\alpha(n) := \pi_n(\alpha)$. For $\alpha\in \varprojlim \beta H_n$, $g\in G$, $m\in \mathbb{N}$, and $S\subseteq H_m$, we have
\begin{align*}
\\[-5 mm]
S\in \alpha g(m) \Leftrightarrow \{x\in H_n: x\circ g|_m\in S\}\in \alpha(n)\\[-5 mm]
\end{align*}
where $n\geq m$ is large enough so that $\ran{g|_m}\subseteq \mathbf{A}_n$. Notice that if $g|_m = h|_m = f$, then $\alpha g(m) = \alpha h(m) := \alpha\cdot f := \lambda_m^\alpha(f)$. By distinguishing the point $1\in \varprojlim \beta H_n$ with $1(m)$ principal on $i_m$, we endow $S(G)$ with the structure of a $G$-ambit, and $(S(G), 1)$ is the greatest ambit (see Theorem 6.3 in [Z]). 

Using the universal property of the greatest ambit, we can define a left-topological semigroup structure on $S(G)$: Given $\alpha$ and $\gamma$ in $\varprojlim \beta H_n$, $m\in \mathbb{N}$, and $S\subseteq H_m$, we have
\begin{align*}
\\[-5 mm]
S\in \alpha\gamma(m) \Leftrightarrow \{f\in H_m: S\in \alpha\cdot f\}\in \gamma(m).\\[-5mm]
\end{align*}
If $\alpha\in S(G)$ and $S\subseteq H_m$, a useful shorthand is to put 
\begin{align*}
\\[-5 mm]
\alpha^{-1}(S) = \{f\in H_m: S\in \alpha\cdot f\}.\\[-5 mm]
\end{align*}
Then the semigroup multiplication can be written as 
\begin{align*}
\\[-5 mm]
S\in \alpha\gamma(m) \Leftrightarrow \alpha^{-1}(S)\in \gamma(m).\\[-5 mm]
\end{align*}
Notice that for fixed $\alpha$, $\alpha\gamma(m)$ depends only on $\gamma(m)$; indeed, if $\alpha\in \varprojlim \beta H_n$, $p\in \beta H_m$, and $S\subseteq H_m$, we have $S\in \alpha\cdot p$ iff $\alpha^{-1}(S)\in p$. In fact, $\alpha\cdot p = \tilde\lambda_m^\alpha(p)$, where the map $\tilde\lambda_m^\alpha$ is the continuous extention of $\lambda_m^\alpha$ to $\beta H_m$. 

As promised in the introduction, we now explain the reason behind our left-right conventions. The primary reason behind considering \emph{right} $G$-flows is because for $G = \aut{\mathbf{K}}$, the \emph{left} uniformity is very natural to describe. Namely, every entourage contains an entourage of the form $\{(g,h)\in G\times G: g|_m = h|_m\}$. This leads naturally to considering the embeddings \emph{from} $\mathbf{A}_m$ to $\mathbf{K}$. If we wanted to consider the \emph{right} uniformity, we would instead be considering partial isomorphisms of $\mathbf{K}$ with \emph{range} $\mathbf{A}_m$, which are less easily described.
\vspace{5 mm}

\section{KPT correspondence}
\vspace{5 mm}

In this section, we provide a brief review of KPT correspondence. For proofs of the results in this section, see [KPT], [NVT], or [Z].

Let $L$ be a relational language and $L^* = L\cup \mathcal{S}$, where $\mathcal{S} = \{S_i: i\in\mathbb{N}\}$ and the $S_i$ are new relational symbols of arity $n_i$. If $\mathbf{A}$ is an $L^*$-structure, write $\mathbf{A}|_L$ for the structure obtained by throwing away the interpretations of the relational symbols in $L^*\setminus L$. If $\mathcal{K}^*$ is a class of $L^*$-structures, set $\mathcal{K}^*|_L = \{\mathbf{A}^*|_L: \mathbf{A}^*\in \mathcal{K}^*\}$. If $\mathcal{K} = \mathcal{K}^*|_L$ and $\mathcal{K}^*$ is closed under isomorphism, we say that $\mathcal{K}^*$ is an \emph{expansion} of $\mathcal{K}$. If $\mathbf{A}^*\in\mathcal{K}^*$ and $\mathbf{A}^*|_L = \mathbf{A}$, then we say that $\mathbf{A}^*$ is an expansion of $\mathbf{A}$, and we write $\mathcal{K}^*(\mathbf{A})$ for the set of expansions of $\mathbf{A}$ in $\mathcal{K}^*$. If $f \in \emb{\mathbf{A}, \mathbf{B}}$ and $\mathbf{B}^*\in \mathcal{K}^*(\mathbf{B})$, we let $\mathbf{B}^*\!\cdot\! f$ be the unique expansion of $\mathbf{A}$ so that $f\in \emb{\mathbf{B}^*\!\cdot\! f, \mathbf{B}^*}$. The expansion $\mathcal{K}^*$ is \emph{precompact} if for each $\mathbf{A}\in \mathcal{K}$, the set $\{\mathbf{A}^*\in\mathcal{K}^*: \mathbf{A}^*|_L = \mathbf{A}\}$ is finite. 

If $\mathcal{K}^*$ is an expansion of the Fra\"iss\'e class $\mathcal{K}$, we say that the pair $(\mathcal{K}^*, \mathcal{K})$ is \emph{reasonable} if for any $\mathbf{A},\mathbf{B}\in\mathcal{K}$, embedding $f: \mathbf{A}\rightarrow \mathbf{B}$, and expansion $\mathbf{A}^*$ of $\mathbf{A}$, then there is an expansion $\mathbf{B}^*$ of $\mathbf{B}$ with $f: \mathbf{A}^*\rightarrow \mathbf{B}^*$ an embedding. When $\mathcal{K}^*$ is also a Fra\"iss\'e class, we have the following equivalent definition.
\vspace{3 mm}

\begin{prop}
\label{Reasonable}
Let $\mathcal{K}^*$ be a Fra\"iss\'e expansion class of the Fra\"iss\'e class $\mathcal{K}$ with \fr limits $\mathbf{K}^*,\mathbf{K}$ respectively. Then the pair $(\mathcal{K}^*,\mathcal{K})$ is reasonable iff $\mathbf{K}^*|_L \cong \mathbf{K}$.
\end{prop}
\vspace{3 mm}

Set $\fin{\mathbf{K}} = \{\mathbf{A}\in \mathcal{K}: \mathbf{A}\subseteq \mathbf{K}\}$. Suppose $(\mathcal{K}^*,\mathcal{K})$ is reasonable and precompact. Set
\begin{align*}
\\[-5 mm]
X_{\mathcal{K}^*} := \{\langle \mathbf{K},\vec{S}\rangle: \langle \mathbf{A}, \vec{S}|_A\rangle \in \mathcal{K}^* \text{ whenever } \mathbf{A}\in \fin{\mathbf{K}}\}.\\[-5 mm]
\end{align*}
We topologize this space by declaring the basic open neighborhoods to be of the form $N(\mathbf{A}^*) := \{\mathbf{K}'\in X_\mathcal{K}^*: \mathbf{A}^*\subseteq \mathbf{K}'\}$, where $\mathbf{A}^*$ is an expansion of some $\mathbf{A}\in\fin{\mathbf{K}}$. We can view $X_{\mathcal{K}^*}$ as a closed subspace of 
\begin{align*}
\\[-5 mm]
\prod_{\mathbf{A}\in \fin{\mathbf{K}}} \{\mathbf{A}^*: \mathbf{A}^*\in \mathcal{K}^*(\mathbf{A})\}.\\[-5 mm]
\end{align*}
Notice that since $(\mathcal{K}^*,\mathcal{K})$ is precompact, $X_{\mathcal{K}^*}$ is compact. If $\bigcup_n \mathbf{A}_n = \mathbf{K}$ is an exhaustion, a compatible metric is given by 
\begin{align*}
\\[-5 mm]
d(\langle \mathbf{K}, \vec{S}\rangle, \langle \mathbf{K}, \vec{T}\rangle = 1/k(\vec{S}, \vec{T}),\\[-5 mm]
\end{align*}
where $k(\vec{S}, \vec{T})$ is the largest $k$ for which $\langle \mathbf{A}_k, \vec{S}|_{A_k}\rangle\cong \langle \mathbf{A}_k, \vec{T}|_{A_k}\rangle$.

We can now form the (right) logic action of $G = \aut{\mathbf{K}}$ on $X_{\mathcal{K}^*}$ by setting $\mathbf{K}'\cdot g$ to be the structure where for each relation symbol $S\in \mathcal{S}$, we have
\begin{align*}
\\[-5 mm]
S^{(\mathbf{K}'\cdot g)}(x_1,...,x_n) \Leftrightarrow S^{\mathbf{K}'}(g(x_1),...,g(x_n)).\\[-5 mm]
\end{align*}
This action is jointly continuous, turning $X_{\mathcal{K}^*}$ into a $G$-flow. For readers used to left logic actions, acting on the right by $g$ is the same as acting on the left by $g^{-1}$.

First let us consider when $X_{\mathcal{K}^*}$ is a minimal $G$-flow. 
\vspace{3 mm}

\begin{defin}
\label{ExpP}
We say that the pair $(\mathcal{K}^*,\mathcal{K})$ has the \emph{Expansion Property} (ExpP) when for any $\mathbf{A}^*\in\mathcal{K}^*$, there is $\mathbf{B}\in\mathcal{K}$ such that for any expansion $\mathbf{B}^*$ of $\mathbf{B}$, there is an embedding $f: \mathbf{A}^*\rightarrow \mathbf{B}^*$.
\end{defin}
\vspace{0 mm}

\begin{prop}
\label{MinimalExpP}
Let $\mathcal{K}^*$ be a reasonable, precompact Fra\"iss\'e expansion class of the Fra\"iss\'e class $\mathcal{K}$ with \fr limits $\mathbf{K}^*, \mathbf{K}$ respectively. Let $G = \aut{\mathbf{K}}$. Then the $G$-flow $X_{\mathcal{K}^*}$ is minimal iff the pair $(\mathcal{K}^*,\mathcal{K})$ has the ExpP.
\end{prop}
\vspace{3 mm}

Expansion classes are particularly inetesting when $\mathcal{K}^*$ has the following combinatorial property.
\vspace{3 mm}

\begin{defin}
\label{RP}
Let $\mathcal{C}$ be a class of finite structures. 
\vspace{2 mm}

\begin{enumerate}
\item
We say that $\mathbf{A}\in \mathcal{C}$ is a \emph{Ramsey object} if for any $r\geq 2$ and any $\mathbf{B}\in \mathcal{C}$ with $\mathbf{A} \leq \mathbf{B}$, there is $\mathbf{C}\in \mathcal{C}$ with $\mathbf{B}\leq \mathbf{C}$ so that for any coloring $c: \emb{\mathbf{A}, \mathbf{C}}\rightarrow r$, there is $h\in \emb{\mathbf{B}, \mathbf{C}}$ with $|c(h\circ \emb{\mathbf{A}, \mathbf{B}})| = 1$. 
\vspace{2 mm}

\item
We say that $\mathcal{C}$ has the Ramsey Property (RP) if every $\mathbf{A}\in \mathcal{C}$ is a Ramsey object. 
\end{enumerate}
\end{defin}
\vspace{3 mm}

The following is one of the major theorems in [KPT]. This theorem in its full generality is proven in [NVT].
\vspace{3 mm}

\begin{theorem}
\label{KPTCorresp}
Let $\mathcal{K}^*$ be a reasonable, precompact Fra\"iss\'e expansion class of the Fra\"iss\'e class $\mathcal{K}$ with \fr limits $\mathbf{K}^*,\mathbf{K}$, respectively. Let $G = \aut{\mathbf{K}}$. Then $X_{\mathcal{K}^*}\cong M(G)$ iff the pair $(\mathcal{K}^*,\mathcal{K})$ has the ExpP and $\mathcal{K}^*$ has the RP.
\end{theorem}
\vspace{3 mm}

Pairs $(\mathcal{K}^*,\mathcal{K})$ of Fra\"iss\'e classes which are reasonable, precompact, satisfy the ExpP, and where $\mathcal{K}^*$ has the RP are called \emph{excellent}. In particular, if $\mathbf{K} = \flim{\mathcal{K}}$, $G = \aut{\mathbf{K}}$, and there is an expansion class $\mathcal{K}^*$ so that $(\mathcal{K}^*, \mathcal{K})$ is excellent, then $M(G)$ is metrizable. The following converse is one of the major theorems of [Z].
\vspace{3 mm}

\begin{theorem}
\label{KPTConverse}
Let $\mathcal{K}$ be a \fr class with $\mathbf{K} = \flim{\mathcal{K}}$ and $G = \aut{\mathbf{K}}$. If $M(G)$ is metrizable, then there is an expansion class $\mathcal{K}^*$ so that $(\mathcal{K}^*, \mathcal{K})$ is excellent.
\end{theorem}
\vspace{5 mm}

\section{Ellis's problem for random relational structures}
\vspace{5 mm}

In this section, we prove Theorem \ref{PestovSolution}. Let $G = \aut{\mathbf{K}}$ for some \fr structure $\mathbf{K} = \bigcup_n \mathbf{A}_n$. If $T\subseteq H_m$, $n\geq m$, and $f\in \emb{\mathbf{A}_m, \mathbf{A}_n}$, set 
\begin{align*}
\\[-5 mm]
f(T) = \{s\in H_n: s\circ f\in T\}.\\[-5 mm]
\end{align*}
This is a minor abuse of notation for two reasons. First, $f$ already denotes a map from $\mathbf{A}_m$ to $\mathbf{A}_n$. Second, for any $N\geq n$, we have $f\in \emb{\mathbf{A}_m, \mathbf{A}_n}\subseteq \emb{\mathbf{A}_m, \mathbf{A}_N}$, so to understand what is meant by $f(T)$ for $T\subseteq H_m$, the intended range $\mathbf{A}_n$ must be understood from context. 

We will freely identify $\mathcal{P}(H_m)$ with $2^{H_m}$; in particular, $G$ acts on $2^{H_m}$ by right shift, where for $\phi\in 2^{H_m}$, $f\in H_m$, and $g\in G$, we have $\phi\cdot g(f) = \phi(g\cdot f)$.
\vspace{3 mm} 

\begin{defin}
\label{MinimalSet}
We call a subset $S\subseteq H_m$ \emph{minimal} if the flow $\overline{\chi_S\cdot G}\subseteq 2^{H_m}$ is minimal.
\end{defin}
\vspace{3 mm}

The formulation of Ellis's problem we will work with is the one concerning retractions given by Pestov. We will be interested in whether every pair $x\neq y\in S(G)$ can be separated by retractions. A characterization of when this occurs for discrete groups can be found in [Ba] (see Proposition 11). We first prove a similar characterization for automorphism groups in the next two lemmas.

Before proceeding, a quick remark on notation is in order. If $X$ is a $G$-flow, then there is a unique map of ambits $\phi: S(G)\rightarrow E(X)$. If $x\in X$ and $p\in S(G)$, we write $x\cdot p$ for $x\cdot \phi(p)$.
\vspace{3 mm}

\lemma
\label{MinimalSeparate}
Suppose $\alpha, \gamma\in S(G)$ cannot be separated by retractions, and let $S\subseteq H_m$ be minimal. Then $S\in \alpha(m)\Leftrightarrow S\in \gamma(m)$.
\vspace{1 mm}

\proof
Let $M\subseteq S(G)$ be a minimal subflow, and consider the non-empty closed subsemigroup $\{p\in M: \chi_S\cdot p = \chi_S\}$. By Ellis's theorem, let $u\in M$ be an idempotent with $\chi_S\cdot u = \chi_S$. As the left multiplication $\lambda_u: S(G)\rightarrow M$ is a retraction, we must have $u\cdot \alpha = u\cdot \gamma$. Therefore $\chi_S\cdot \alpha = \chi_S\cdot \gamma$, and $\alpha^{-1}(S) = \gamma^{-1}(S) := T$. It follows that $i_m\in T$ iff $S\in \alpha(m)$ iff $S\in \gamma(m)$.
\qedhere
\vspace{5 mm}

For each $m< \omega$, let $B_m\subseteq \mathcal{P}(H_m)$ be the Boolean algebra generated by the minimal subsets of $H_m$. Let $B_m'$ be the Boolean algebra $\{T\subseteq H_m: \exists n\geq m (i^n_m(T)\in B_n)\}$. 
\vspace{3 mm}

\lemma
\label{MinimalBA}
Fix $M\subseteq S(G)$ a minimal subflow. The following are equivalent.
\vspace{1 mm}
\begin{enumerate}
\item
Retractions of $S(G)$ onto $M$ separate points of $S(G)$,
\vspace{2 mm}

\item
For every $m< \omega$, we have $B_m' = \mathcal{P}(H_m)$.
\end{enumerate}
\vspace{1 mm}

\proof
$\neg(1)\Rightarrow \neg(2)$ Suppose that there are $\alpha\neq \gamma \in S(G)$ which cannot be separated by retractions. Find $m< \omega$ with $\alpha(m)\neq \gamma(m)$, and find $T\subseteq H_m$ with $T\in \alpha(m)$, $T\not\in \gamma(m)$. Note that for every $n\geq m$, we have $i^n_m(T)\in \alpha(n)$ and $i^n_m(T)\not\in \gamma(n)$. Towards a contradiction, suppose for some $n\geq m$ that $i^n_m(T)$ was a Boolean combination of minimal sets $A_1,...,A_k\subseteq H_n$. By Lemma \ref{MinimalSeparate}, $\alpha(n)$ and $\gamma(n)$ agree on the membership of each $A_i$, hence also on the membership of $T$, a contradiction.
\vspace{3 mm}

$\neg(2)\Rightarrow \neg(1)$ Suppose that $T\subseteq H_m$ is not in the Boolean algebra $B_m'$. Let $\mathrm{St}(B_m')$ denote the Stone space of $B_m'$. Since $T\not\in B_m'$, we can find $q\in \mathrm{St}(B_m')$ so that every $S\in q$ has $S\cap T\neq \emptyset$ and $S\setminus T = \emptyset$. Form the inverse limit $\varprojlim \mathrm{St}(B_n')$, and find $p\in \varprojlim \mathrm{St}(B_n')$ with $p(m) = q$. Then find $\alpha, \gamma\in \varprojlim \beta H_n$ with $T\in \alpha(m)$, $T\not\in \gamma(m)$ which both extend $p$. Then $\alpha$ and $\gamma$ cannot be separated by retractions.
\qedhere
\vspace{5 mm}

Notice that item (2) of Lemma \ref{MinimalBA} does not depend on $M$. In general, the relation of whether $x\neq y\in S(G)$ can be separated by retractions does not depend on the minimal subflow of $S(G)$ chosen, but we postpone this discussion until the end of section 6 (see the discussion after Theorem \ref{ExistsNotClosed}).

Now suppose that $(\mathcal{K}^*, \mathcal{K})$ is an excellent pair of \fr classes. Given a set of expansions $E\subseteq \mathcal{K}^*(\mathbf{A}_m)$ and $\mathbf{K}'\in X_{\mathcal{K}^*}$, let $H_m(E, \mathbf{K}') = \{f\in H_m: \mathbf{K}'\cdot f \in E\}$. If $n\geq m$ and $\mathbf{A}_n'\in \mathcal{K}^*(\mathbf{A}_n)$, we set $\emb{E, \mathbf{A}_n'} = \{f\in \emb{\mathbf{A}_m, \mathbf{A}_n}: \mathbf{A}_n'\cdot f\in E\}$.
\vspace{3 mm}

\begin{prop}
\label{Minimal}
$S\subseteq H_m$ is minimal iff there is $\mathbf{K}'\in X_{\mathcal{K}^*}$ and $E\subseteq \mathcal{K}^*(\mathbf{A}_m)$ so that $S = H_m(E, \mathbf{K}')$
\end{prop}
\vspace{1 mm}

\proof
One direction is easy once we note that given $E\subseteq \mathcal{K}^*(\mathbf{A}_m)$, the map $\phi_E: X_{\mathcal{K}^*}\rightarrow 2^{H_m}$ given by $\phi_E(\mathbf{K}') = H_m(E, \mathbf{K}')$ is a map of $G$-flows. In the other direction, let $S\subseteq H_m$ be minimal. Then $Y := \overline{\chi_S\cdot G}\subseteq 2^{H_m}$ is a minimal $G$-flow, so fix a $G$-map $\phi: X_{\mathcal{K}^*}\rightarrow Y$. Note that $\phi(\mathbf{K}^*)$ must be a $G^*$-fixed point, so by ultrahomogeneity of $\mathbf{K}^*$ must be of the form $H_m(E, \mathbf{K}^*)$ for some $E\subseteq \mathcal{K}^*(\mathbf{A}_m)$. It follows that $\phi = \phi_E$, so in particular $S = H_m(E, \mathbf{K}')$ for some $\mathbf{K}'\in X_{\mathcal{K}^*}$.
\qedhere
\vspace{4 mm}

The main tool allowing us to prove Theorem \ref{PestovSolution} is an explicit characterization of $M(G)$ for certain autormorphism groups $G$. The following facts can be found in [KPT].
\vspace{3 mm}

\begin{fact}[\cite{KPT}, Theorem 8.1]
Let $\mathcal{K} = \age{\mathbf{K}}$, where $\mathbf{K}$ is any of the structures in the statement of Theorem \ref{PestovSolution}. Let $\mathcal{K}^*$ be the class of linearly ordered members of $\mathcal{K}$. Then $(\mathcal{K}^*, \mathcal{K})$ is an excellent pair. Setting $G = \aut{\mathbf{K}}$, then $M(G) = X_{\mathcal{K}^*}$ is the space of linear orders of $\mathbf{K}$.
\end{fact}
\vspace{3 mm}

The next theorem is the simplest case of Theorem \ref{PestovSolution}. The following notion will be useful in the proof. Given $T\subseteq H_m$ and $N\geq m$, an \emph{$N$-pattern} of $T$ is a set $S\subseteq \emb{\mathbf{A}_m, \mathbf{A}_N}$ so that there is $y\in H_N$ with $S = \{f\in \emb{\mathbf{A}_m, \mathbf{A}_N}: y\circ f\in T\}$. 
\vspace{3 mm}

\begin{theorem}
\label{AdmissibleSets}
Let $\mathbf{K} = \{x_n: n<\omega\}$ be a countable set with no structure \rm(so $G\cong S_\infty$\rm), and set $\mathbf{A}_m = \{x_i: i< m\}$. Then for every $m\geq 2$, $B_m'\subseteq 2^{H_m}$ is meager. In particular, any $T\subseteq H_m$ belonging to the dense $G_\delta$ set $\{T\subseteq H_m: \overline{\chi_T\cdot G} = 2^{H_m}\}$ is not in $B_m'$.
\end{theorem}
\vspace{1 mm}

\proof
Let $T\subseteq H_m$ have dense orbit. So for every $N\geq m$, every $S\subseteq \emb{\mathbf{A}_m, \mathbf{A}_N}$ is an $N$-pattern of $T$. Towards a contradiction, suppose for some $n\geq m$ that $i^n_m(T)$ was a Boolean combination of minimal sets $S_1,...,S_k\subseteq H_n$. Let $N \gg n$; we will obtain a contradiction by counting the number of $N$-patterns in $i^n_m(T)$, which by assumption is $2^{m!\binom{N}{m}}\geq 2^{N^m/2}$. 

Since $|\mathcal{K}^*(\mathbf{A}_n)| = n!$ and since there are $N!$ linear orders on $\mathbf{A}_N$, this gives us $2^{n!}N!$ possible $N$-patterns for each $S_i$ by Proposition \ref{Minimal}. Therefore any $N$-pattern of $T$ must be a Boolean combination of some $k$ of these $N$-patterns. Each choice of $k$ patterns results in at most $2^{2^k}$ Boolean combinations, so the total number of possible patterns is at most $2^{2^k}(2^{n!}N!)^k$ patterns. Noting that $n$ and $k$ remain fixed as we let $N$ grow large, we have that asymptotically there are fewer than $N^{kN}$ possible $N$-patterns of $i^n_m(T)$, which is far less than $2^{N^m/2}$, a contradiction.
\qedhere
\vspace{4 mm}

We now consider the case where $\mathbf{K} = \flim{\mathcal{K}}$ is the random $r$-uniform hypergraph or the random for some $r\geq 2$. In order to generalize the arguments in the proof of Theorem \ref{AdmissibleSets}, we will need some control over the exhaustion $\mathbf{K} = \bigcup_n \mathbf{A}_n$. We will do this by not specifying an exhaustion in advance, but instead determining parts of it as we proceed.

We will need the following notion. With $\mathcal{K}$ as above, let $\mathbf{A}\subseteq \mathbf{B}\in \mathcal{K}$, and let $\mathbf{C}\subseteq \mathbf{D}\in \mathcal{K}$. We say that $\mathbf{D}$ \emph{extends $\mathbf{C}$ along $\mathbf{A}\subseteq \mathbf{B}$} if for any $f\in \emb{\mathbf{A}, \mathbf{C}}$, there is an $h\in \emb{\mathbf{B}, \mathbf{D}}$ with $h|_\mathbf{A} = f$. 

Given $\mathbf{C}\in \mathcal{K}$, write $|\mathbf{C}|$ for the number of vertices in $\mathbf{C}$.
\vspace{3 mm}

\begin{lemma}
\label{Extension}
Let $\mathcal{K}$ be the class of $r$-uniform hypergraphs for some $r\geq 2$. Let $e\subseteq \mathbf{B}\in \mathcal{K}$, where $e\in \mathcal{K}$ is the hypergraph on $r$ vertices consisting of an edge, and let $\mathbf{C}\in \mathcal{K}$ with $|\mathbf{C}| = N$. Then there is $\mathbf{D}\in \mathcal{K}$ extending $\mathbf{C}$ along $e\subseteq \mathbf{B}$ with $|\mathbf{D}| \leq cN^{r-1}$ for some constant $c$ depending only on $|\mathbf{B}|$.
\end{lemma}
\vspace{1 mm}

\proof
Recall that given an $r$-uniform hypergraph $\mathbf{C}$, a \emph{matching} is a subset of the edges of $\mathbf{C}$ so that each vertex is included in at most one edge. By Baranyai's theorem [B], the edge set of $\mathbf{C}$ can be partitioned into $M_1,...,M_\ell$ with each $M_i$ a matching and with $\ell \leq \binom{r\lceil N/r\rceil}{r}/\lceil N/r\rceil) \approx c_0N^{r-1}$. For each $i\leq \ell$ and $j\leq \ell!$, let $D_i^j$ be a set of $|\mathbf{B}|-r$ new vertices. We will define the hypergraph $\mathbf{D}$ on vertex set $\mathbf{C}\cup \bigcup_{i\leq \ell}\bigcup_{j\leq r!} D_i^j$. First add edges to $\mathbf{D}$ so that $D_i^j \cong \mathbf{B}\setminus e$. For each $e'\in M_i$, enumerate the embeddings $f_j: e\rightarrow \mathbf{C}$ with range $e'$. Add edges to $\mathbf{D}$ so that each $f_j$ extends to an embedding $h_j: \mathbf{B}\rightarrow \mathbf{D}$ with range $e'\cup D_i^j$. This is possible since each $M_i$ is a matching. The hypergraph $\mathbf{D}$ has $|\mathbf{D}|\leq cN^{r-1}$ as desired.
\qedhere
\vspace{3 mm}

\begin{theorem}
\label{Hypergraphs}
Let $\mathcal{K}$ be the class of $r$-uniform hypergraphs for $r\geq 2$, with $\mathbf{K} = \flim{\mathcal{K}}$. Let $\mathbf{A}_r\subseteq \mathbf{K}$ be an edge on $r$ vertices. Then if $T\subseteq H_r$ has dense orbit, then $T\not\in B_r'$.
\end{theorem}
\vspace{-2 mm}

\begin{rem}
Though we are not specifying an exhaustion in advance, we will still use some of the associated notation. In particular, when we write $\mathbf{A}_m$ for some $m< \omega$, we mean a subgraph of $\mathbf{K}$ on $m$ vertices. 
\end{rem}
\vspace{-2 mm}

\proof
Suppose towards a contradiction that there were some graph $\mathbf{A}_n\supseteq \mathbf{A}_r$ so that $i^n_r(T)$ was a Boolean combination of minimal sets $B_1,...,B_k\subseteq H_n$. Let $N\gg n$, and fix a graph $\mathbf{A}_N\supseteq \mathbf{A}_n$ with at least $\binom{N}{r}/2$ edges. Let $\mathbf{A}_{N'}\supseteq \mathbf{A}_N$ extend $\mathbf{A}_N$ along $\mathbf{A}_r\subseteq \mathbf{A}_n$ with $N'\approx cN^{r-1}$ as guaranteed by Lemma \ref{Extension}. We will obtain a contradiction by counting the number of $N'$-patterns in $i^n_r(T)$. Exactly as in the proof of Theorem \ref{AdmissibleSets}, there are fewer than $2^{2^k}(2^{n!}N'!)^k < (N')^{kN'}\approx (cN^{r-1})^{kcN^{r-1}}$ many $N'$-patterns. But since $T$ has dense orbit, there must be at least $2^{\binom{N}{r}/2} > 2^{dN^r}$ $N'$-patterns in $i^n_r(T)$ for $d$ some constant, a contradiction.
\qedhere
\vspace{4 mm}

We next turn to the class $\mathcal{K}$ of $K_r$-free graphs for some $r\geq 3$. We will need a result similar to Lemma \ref{Extension}, but the given proof will not work as the construction doesn't preserve being $K_r$-free. 
\vspace{3 mm}

\begin{lemma}
\label{ForbExtension}
Let $\mathcal{K}$ be the class of $K_r$-free graphs for some $r\geq 3$. Let $e\subseteq \mathbf{B}\in \mathcal{K}$, where $e$ is an edge, and let $\mathbf{C}\in \mathcal{K}$ with $|\mathbf{C}| = N$. Then there is $\mathbf{D}\in \mathcal{K}$ extending $\mathbf{C}$ along $e\subseteq \mathbf{B}$ with $|\mathbf{D}|\leq cN^{2(r-1)/(r-2)}$.
\end{lemma}
\vspace{1 mm}

\proof
Let $R(r, n)$ be the Ramsey number of $r$ and $n$. In [AKS], it is shown that $R(r, n) = o(n^{r-1})$. Since $\mathbf{C}$ is $K_r$-free, this implies that $\mathbf{C}$ has an independent set of size at least $N^{1/(r-1)}$. By repeatedly removing independent sets, we see that the chromatic number of $\mathbf{C}$ is at most $\ell \approx \frac{r-1}{r-2}N^{(r-2)/(r-1)}$; one can see this by solving the differential equation $dy/dt = -y^{1/(r-1)}$ and setting $y(0) = N$.

Write $\mathbf{C} = C_1\sqcup\cdots \sqcup C_\ell$ so that each $C_i$ is an independent set. For every ordered pair $(i, j)$ of distinct indices with $i, j\leq \ell$, let $D_{(i,j)}$ be a set of $|\mathbf{B}|-2$ new vertices. We will define the graph $\mathbf{D}$ on vertex set $\mathbf{C}\cup\bigcup_{\{i,j\}\in [\ell]^2} (D_{(i,j)}\cup D_{(j,i)})$. First add edges to $\mathbf{D}$ so that $D_{(i,j)}\cong \mathbf{B}\setminus e$; fix $h': \mathbf{B}\setminus e\rightarrow D_{(i,j)}$ an isomorphism. Write $e = \{a,b\}$; if $f: e\rightarrow \mathbf{C}$ with $f(a) = i$ and $f(b) = j$, then add edges to $\mathbf{D}$ so that $h'\cup f := h: \mathbf{B}\rightarrow \mathbf{D}$ is an embedding with range $f(e)\cup D_{(i,j)}$. The graph $\mathbf{D}$ is $K_r$-free and has $|\mathbf{D}|\leq cN^{2(r-2)/(r-1)}$ as desired.
\qedhere
\vspace{3 mm}

\begin{theorem}
\label{ForbGraphs}
Let $\mathcal{K}$ be the class of $K_r$-free graphs for some $r\geq 3$, with $\mathbf{K} = \flim{\mathcal{K}}$. Let $\mathbf{A}_2\subseteq \mathbf{K}$ be an edge. Then if $T\subseteq H_2$ has dense orbit, then $T\not\in B_2'$.
\end{theorem}

As in the proof of Theorem \ref{Hypergraphs}, we will not specify an exhaustion in advance, but we will still use some of the notational conventions.

\proof
Suppose towards a contradiction that there were some graph $\mathbf{A}_n\supseteq \mathbf{A}_2$ so that $i^n_2(T)$ was a Boolean combination of minimal sets $B_1,...,B_k\subseteq H_n$. Let $N\gg n$, and fix a graph $\mathbf{A}_N\supseteq \mathbf{A}_n$ with at least $\binom{N}{2}/r$ edges. Let $\mathbf{A}_{N'}\supseteq \mathbf{A}_N$ extend $\mathbf{A}_N$ along $\mathbf{A}_2\subseteq \mathbf{A}_n$ with $N'\approx cN^{2(r-2)/(r-1)}$ as guaranteed by Lemma \ref{ForbExtension}. We now obtain a contradiction by counting $N'$-patterns. Once again, there are fewer than $(N')^{kN'}$ many $N'$-patterns in $i^n_2(T)$, which contradicts the fact that there are at least $2^{\binom{N}{2}/r}$ many $N'$-patterns.
\qedhere
\vspace{4 mm}

We end this section with a conjecture. While it is a strict sub-conjecture of Conjecture \ref{PestovCon}, we think it might be more easily approached.
\vspace{3 mm}

\begin{conj}
\label{NewConj}
Let $G$ be a closed, non-compact subgroup of $S_\infty$ with metrizable universal minimal flow. Then $S(G)\not\cong E(M(G))$.
\end{conj}
\vspace{5 mm}

\section{A closer look at $S_\infty$}
\label{SInfty}
\vspace{5 mm}

In this section, we take a closer look at $S(S_\infty)$, with an eye towards understanding which pairs of points $x\neq y\in S(S_\infty)$ can be separated by retractions. We view $S_\infty$ as the group of permutations of $\omega$. We can view $\omega$ as a \fr structure in the empty language. We set $\mathbf{A}_n = n$, so that $H_n$ is the set of all injections from $n$ into $\omega$, and for $m\leq n$, $\emb{\mathbf{A}_m, \mathbf{A}_n}$ is the set of all injections from $m$ into $n$. We will often abuse notation and write $s\in H_m$ as the tuple $(s_0,...,s_{m-1})$, where $s_i = s(i)$. 

We start by developing some notions for any automorphism group. Let $f\in \emb{\mathbf{A}_m, \mathbf{A}_n}$. If $\mathcal{F}\subseteq \mathcal{P}(H_m)$ is a filter, then we write $f(\mathcal{F})$ for the filter generated by $\{f(T): T\in \mathcal{F}\}$. If $\mathcal{H}\subseteq \mathcal{P}(H_n)$ is a filter, then $\tilde{f}(\mathcal{H})$ is the push-forward filter $\{T\subseteq H_m: f(T)\in \mathcal{H}\}$. This may seem like a conflict of notation since $\tilde{f}: \beta H_n\rightarrow \beta H_m$ is the extended dual map of $f$. We can justify this notation as follows. To each filter $\mathcal{H}$ on $H_n$, we associate the closed set $X_\mathcal{H}:= \bigcap_{A\in \mathcal{H}} \overline A\subseteq \beta H_n$. Conversely, given a closed set $X\subseteq \beta H_n$, we can form the \emph{filter of clopen neighborhoods} $\mathcal{F}_X := \{A\subseteq H_n: X\subseteq \overline A\}$. Then we obtain the identity 
\begin{align*}
\\[-5 mm]
X_{\tilde{f}(\mathcal{H})} = \tilde{f}(X_\mathcal{H}).\\[-5 mm]
\end{align*}
A similar identity holds given a filter $\mathcal{F}$ on $H_m$:
\begin{align*}
\\[-5 mm]
X_{f(\mathcal{F})} = \tilde{f}^{-1}(X_\mathcal{F}).\\[-5 mm]
\end{align*}

Let $Y\subseteq S(G)$ be closed. Let $\pi_m: S(G)\rightarrow \beta H_m$ be the projection map. Then $\pi_m(Y)$ is a closed subset of $\beta H_m$. Write $\mathcal{F}_m^Y$ for $\mathcal{F}_{\pi_m(Y)}$. For $n\geq m$, the filter $\mathcal{F}_n^Y$ extends the filter $i_m^n(\mathcal{F}_m^Y)$ and $\tildei{m}{n}(\mathcal{F}_n^Y) = \mathcal{F}_m^Y$. Conversely, given filters $\mathcal{F}_m$ on $H_m$ for every $m< \omega$ such that $\mathcal{F}_n$ extends $i_m^n(\mathcal{F}_m)$ and with $\tildei{m}{n}(\mathcal{F}_n) = \mathcal{F}_m$, there is a unique closed $Y\subseteq S(G)$ with $\mathcal{F}_m = \mathcal{F}_m^Y$ for each $m< \omega$. We will call such a sequence of filters \emph{compatible}.

We will need to understand the filters $\mathcal{F}_m^M$ when $M\subseteq S(G)$ is a minimal subflow. It turns out that these filters are characterized by a certain property of their members. 
\vspace{3 mm}

\begin{defin}
Given $T\subseteq H_m$, we say that $T$ is \emph{thick} if either of the following equivalent items hold (see [Z1]).
\vspace{2 mm}

\begin{enumerate}
\item
$\chi_{H_m}\in \overline{\chi_T\cdot G}$.
\vspace{2 mm}

\item
For every $n\geq m$, there is $s\in H_n$ with $s\circ \emb{A_m, A_n}\subseteq T$. 
\end{enumerate}
\end{defin}
\vspace{3 mm}

We can now state the following fact from [Z1].

\begin{theorem}
\label{MinimalFlows}
Let $G$ be an automorphism group, and let $M\subseteq S(G)$ be closed. Then $M$ is a minimal subflow iff each $\mathcal{F}_m^M$ is a maximal filter of thick sets.
\end{theorem}
\vspace{3 mm}

Another observation is the following.
\begin{prop}
	\label{Invariant}
	Say $Y\subseteq S(G)$ is a subflow, and let $T\in \mathcal{F}_m^Y$ and $f\in \emb{\mathbf{A}_m, \mathbf{A}_n}$. Then $f(T)\in \mathcal{F}_n^Y$.
\end{prop}

\begin{proof}
	Pick $g\in G$ with $g|_m = f$. Then for any $\alpha\in S(G)$, we have $T\in \alpha g(m)$ iff $f(T)\in \alpha(n)$. As $Y$ is $G$-invariant, the result follows.
\end{proof}
\vspace{5 mm}

We now turn our attention to $G = S_\infty$. Let $\{\sigma_i: i< m!\}$ list the permutations of $m$, i.e.\ the members of $\emb{\mathbf{A}_m, \mathbf{A}_m}$. Then by Proposition \ref{Invariant}, $\bigcap_i \sigma_i(T)\in \mathcal{F}_m^M$. Call $S\subseteq H_m$ \emph{saturated} if whenever $(a_0,...,a_{m-1})\in S$ and $\sigma$ is a permutation, then $(a_{\sigma(0)},...,a_{\sigma(m-1)})\in S$. We have just shown that $\mathcal{F}_m^M$ has a base of saturated sets.

Let $\phi: \bigsqcup_n H_n\rightarrow \bigsqcup_n [\omega]^n$ be the order forgetful map, i.e.\ for $(y_0,...,y_{m-1})\in H_m$, we set $\phi(y_0,...,y_{m-1}) = \{y_0,...,y_{m-1}\}\in [\omega]^m$. Any filter $\mathcal{F}$ on $H_m$ pushes forward to a filter $\phi(\mathcal{F})$ on $[\omega]^m$. We can define a \emph{thick} subset of $[\omega]^m$ in a very similar fashion to a thick subset of $H_m$; more precisely, we say $T\subseteq [\omega]^m$ is thick iff for every $n\geq m$, there is $s\in [\omega]^n$ with $[s]^m\subseteq T$. Call $S\subseteq [\omega]^m$ \emph{thin} if it is not thick. We now have the following crucial corollary of Ramsey's theorem: if $T\subseteq [\omega]^m$ is thick and $T = T_0\cupdots T_k$, then some $T_i$ is thick. In particular, if $\mathcal{H}$ is a \emph{thick filter} on $[X]^m$, i.e.\ a filter containing only thick sets, then we can extend $\mathcal{H}$ to a thick ultrafilter. It also follows that for every $m< \omega$, the collection of thin subsets of $[\omega]^m$ forms an ideal.
\vspace{3 mm}

\begin{theorem}
\label{PushUltra}
Let $M\subseteq S(S_\infty)$ be a minimal right ideal. Then for every $m< \omega$, $\phi(\mathcal{F}_m^M)$ is a thick ultrafilter. Conversely, if $p\in \beta [\omega]^m$ is a thick ultrafilter, then $\{\phi^{-1}(T): T\in p\}$ generates a maximal thick filter on $H_m$, hence there is $M\in S(S_\infty)$ with $p = \phi(\mathcal{F}_m^M)$.
\end{theorem}
\vspace{1 mm}

\begin{proof}
Clearly $\phi(\mathcal{F}_m^M)$ is a thick filter. Towards a contradiction, suppose it is not an ultrafilter, and extend it to a thick ultrafilter $p\in \beta [\omega]^m$. Let $T\in p\setminus \phi(\mathcal{F}_m^M)$. Then $\phi^{-1}(T)\not\in \mathcal{F}_m^M$. However, $\phi^{-1}(T)\cap S$ is thick for every saturated $S\in \mathcal{F}_m^M$. As saturated sets form a base for $\mathcal{F}_m^M$, this contradicts the maximality of $\mathcal{F}_m^M$.

Now let $p\in \beta [\omega]^m$ be a thick ultrafilter. Then $\mathcal{F}:= \{\phi^{-1}(T): T\in p\}$ generates a thick filter. Suppose $S\subseteq H_m$ and $\{S\}\cup \mathcal{F}$ generated a thick filter strictly larger than $\mathcal{F}$. We may assume $S$ is saturated. Then $\phi(S)\in p$, so $\phi^{-1}(\phi(S)) = S\in \mathcal{F}$, a contradiction.
\end{proof}
\vspace{5 mm}

Notice that if $p\in \beta [\omega]^n$ is thick and $m\leq n$, then there is a unique thick ultrafilter $q\in \beta [\omega]^m$ with the property that $\{a\in [\omega]^n: [a]^m\subseteq S\}\in p$ for every $S\in q$. Certainly such a $q$ must be unique. To see that this $q$ exists, suppose $[\omega]^m = S\sqcup T$. Then the set $\{a\in [\omega]^n: [a]^m\cap S\neq \emptyset \text{ and } [a]^m\cap T\neq \emptyset\}$ is not thick. We will write $\pi_m^n(p)$ for this $q$. If $M\subseteq S(G)$ is a minimal right ideal and $p = \phi(\mathcal{F}_n^M)$, then we have $\pi_m^n(p) = \phi(\mathcal{F}_m^M)$.

Let $\mathrm{LO}(\omega)$ be the space of linear orders on $\omega$. Viewed as a subset of the right shift $2^{H_2}$, $\mathrm{LO}(\omega)$ becomes an $S_\infty$-flow. It is known (see [GW] or [KPT]) that $\mathrm{LO}(\omega)\cong M(S_\infty)$. Indeed, we saw in section 4 that if $\mathcal{K}$ is the class of finite sets and $\mathcal{K}^*$ is the class of finite linear orders, then $(\mathcal{K}^*, \mathcal{K})$ is an excellent pair, and $X_\mathcal{K^*}\cong \mathrm{LO}(\omega)$. If $M\subseteq S(S_\infty)$ is a minimal right ideal and $<\in \mathrm{LO}(\omega)$, then the map $\lambda: M\rightarrow \mathrm{LO}(\omega)$ given by $\lambda(\alpha) =\, <\!\cdot \alpha:= \lim_{g_i\rightarrow \alpha} <\cdot g_i$ is an $S_\infty$-flow isomorphism. We will often write $<^\alpha$ for $<\cdot \alpha$, and we will write $>$ for the reverse linear order of $<$.

If $<_0, <_1\in \mathrm{LO}(\omega)$ and $m\geq 2$, define the set 
\begin{align*}
\\[-5 mm]
A_m(<_0, <_1) = \{\{a_0,...,a_{m-1}\}\in [\omega]^m: \forall i,j< m(a_i<_0 a_j \Leftrightarrow a_i<_1 a_j)\}\\[-5 mm]
\end{align*}
 and define $B_m(<_0, <_1) = A_m(<_0, >_1)$. If $s\in H_m$, we say that $<_0$ and $<_1$ \emph{agree on $s$} if $\phi(s)\in A_m(<_0, <_1)$, and we say that they \emph{anti-agree on $s$} if $\phi(s)\in B_m(<_0, <_1)$. When $m = 2$ we often omit the subscript. If $M$ is a minimal right ideal, then $\phi(\mathcal{F}_2^M)$ contains exactly one of $A(<_0, <_1)$ or $B(<_0, <_1)$. Let $A^M\subseteq \mathrm{LO}(\omega)\times \mathrm{LO}(\omega)$ be defined $A^M = \{(<_0, <_1): A(<_0, <_1)\in \phi(\mathcal{F}_2^M)\}$. Then $A^M$ is certainly reflexive and symmetric. To see that $A^M$ is an equivalence relation, note that $A(<_0, <_1)\cap A(<_1, <_2)\subseteq A(<_0, <_2)$. Furthermore, $A^M$ has exactly two equivalence classes; this is because $B(<_0, <_1)\cap B(<_1, <_2)\subseteq A(<_0, <_2)$.
\vspace{3 mm}

\lemma
\label{LinearOrders}
Let $M\subseteq S(S_\infty)$ be a minimal right ideal, and let $(<_0, <_1)\in A^M$. Then for any $m< \omega$, we have $A_m(<_0, <_1)\in \phi(\mathcal{F}_m^M)$.
\vspace{1 mm}

\proof
Let $\{f_i: i< k\}$ enumerate $\emb{A_2, A_m}$. Then $\bigcap_i f_i(\phi^{-1}(A(<_0, <_1)))\in \mathcal{F}_m^M$, and this is exactly the desired set.
\qedhere
\vspace{4 mm}

\lemma
\label{Idempotents}
Let $M\subseteq S(S_\infty)$ be a minimal right ideal, and let $\alpha\in M$. Then the following are equivalent:
\vspace{2 mm}

\begin{enumerate}
\item
$\alpha$ is an idempotent,
\vspace{2 mm}

\item
 For any $<\in \mathrm{LO}(\omega)$, we have $(<, <^\alpha)\in A^M$,
\vspace{2 mm}

\item
There is $<\in \mathrm{LO}(\omega)$ with $(<, <^\alpha)\in A^M$.
\end{enumerate}
\vspace{1 mm}

\proof
Suppose $\alpha\in M$ is idempotent, and let $<\in \mathrm{LO}(\omega)$. Then considering $i_2 = (x_0,x_1)\in \emb{A_2, A_2}$, we have $x_0 <^\alpha x_1$ iff $\{f\in H_2: f_0 < f_1\}\in \alpha(2)$. But since $\alpha$ is idempotent, this is equivalent to $\{f\in H_2: f_0 <^\alpha f_1\}\in \alpha(2)$. But this implies that $\phi^{-1}(A(<, <^\alpha))\in \alpha(2)$, implying that $A(<, <^\alpha)\in \phi(\mathcal{F}_2^M)$.

Conversely, suppose $\alpha\in M$ and $<\in \mathrm{LO}(\omega)$ with $(<, <^\alpha)\in A^M$. If $f\in \emb{A_2, A_n}$, then we have $f_0 <^\alpha f_1$ iff $\{s\in H_n: s(f_0)< s(f_1)\}\in \alpha(n)$. By Lemma \ref{LinearOrders}, we see that this is iff $\{s\in H_n: s(f_0) <^\alpha s(f_1)\}\in \alpha(n)$. It follows that $<\cdot \alpha\cdot \alpha = <\cdot\alpha$, so $\alpha$ is idempotent.
\qedhere
\vspace{4 mm}

\theorem
\label{ProductIdemp}
Let $M, N\subseteq S(S_\infty)$ be minimal right ideals. The following are equivalent.
\vspace{2 mm}

\begin{enumerate}
\item
$A^M = A^N$,
\vspace{2 mm}

\item
If $u\in M$ and $v\in N$ are idempotents, then $uv\in M$ is also idempotent.
\end{enumerate}
\vspace{1 mm}

\proof
Suppose $A^M \neq A^N$, with $(<_0, <_1)\in A^N\setminus A^M$. Find $u\in M$ with $<_0\cdot u =\, <_0$, and find $v\in N$ with $<_0\cdot v =\, <_1$. By Lemma \ref{Idempotents}, $u$ and $v$ are idempotents and $uv$ is not an idempotent.

Conversely, suppose $u\in M$ and $v\in N$ are idempotents with $uv$ not idempotent. Find $<_0\,\in \mathrm{LO}(\omega)$ with $<_0\cdot u =\, <_0$, and let $<_1 \,=\, <_0\cdot v$. Since $v$ is idempotent, we have by Lemma \ref{Idempotents} that $(<_0, <_1)\in A^N$; but since $uv$ is not idempotent, we have $(<_0, <_1)\not\in A^M$.
\qedhere
\vspace{6 mm}

It is easy to construct minimal right ideals $M, N\subseteq S(S_\infty)$ with $A^M\neq A^N$. Let $<_0, <_1\in \mathrm{LO}(\omega)$ be linear orders so that for every $m< \omega$, there are $s^m = (s_0^m,...,s_{m-1}^m)\in H_m$ and $t^m = (t_0^m,...,t_{m-1}^m)\in H_m$ so that $<_0$ and $<_1$ agree on $s^m$ and anti-agree on $t^m$. Let $M\subseteq S(S_\infty)$ be a minimal subflow with $\phi^{-1}(A(<_0, <_1))\in \mathcal{F}_2^M$, and let $N\subseteq S(S_\infty)$ be a minimal subflow with $\phi^{-1}(B(<_0, <_1))\in \mathcal{F}_2^N$. Then $(<_0, <_1)\in A^M\setminus A^N$. 
\vspace{3 mm}

We now turn our attention to constructing $M\neq N\subseteq S(S_\infty)$ minimal right ideals with $A^M = A^N$; this will prove Theorem \ref{TwoIdeals} as a corollary of Theorem \ref{ProductIdemp}. To this end, we will construct two thick ultrafilters $p\neq q\in \beta [\omega]^3$ with $\pi_2^3(p) = \pi_2^3(q)$, so that whenever $M$ and $N$ are minimal subflows of $S(S_\infty)$ with $\phi(\mathcal{F}_3^M) = p$ and $\phi(\mathcal{F}_3^N) = q$, then $\phi(\mathcal{F}_2^M) = \phi(\mathcal{F}_2^N)$. In particular, this implies that $A^M = A^N$.

Recall that a \emph{selective ultrafilter} is an ultrafilter $p$ on $\omega$ with the property that for any finite coloring $c: [\omega]^2\rightarrow r$, there is a $p$-large set $A\subseteq \omega$ which is monochromatic for $c$. Another way of saying this is as follows. Given a set $A\subseteq \omega$, set $\lambda A = [A]^2$, and if $\mathcal{F}$ is a filter on $\omega$, let $\lambda \mathcal{F}$ be the filter generated by $\{\lambda A: A\in \mathcal{F}\}$. Then the ultrafilter $p$ is selective iff $\lambda p$ is an ultrafilter. The existence of selective ultrafilters is independent of ZFC.

We will be considering the following generalizations of selective ultrafilters. Let $m< \omega$. If $T\subseteq [\omega]^m$, we set $\lambda T = \{s\in [\omega]^{m+1}: [s]^m\subseteq T\}$. If $n > m$, we set $\lambda^{(n-m)}(T) = \{s\in [\omega]^n: [s]^m\subseteq T\}$. Notice that the $\lambda^{(n-m)}$ operation is the same as applying $\lambda$ $(n-m)$-many times, justifying this notation. If $\mathcal{F}$ is a filter on $[\omega]^m$, we let $\lambda^{(n-m)} \mathcal{F}$ be the filter generated by $\{\lambda^{(n-m)} T: T\in \mathcal{F}\}$. It can happen that for some $T\in \mathcal{F}$ we have $\lambda^{(n-m)} T = \emptyset$. We will usually be working under assumptions that prevent this from happening. For instance, if $T\subseteq [\omega]^m$ is thick, then $\lambda^{(n-m)} T\neq \emptyset$ for every $n> m$. Even better, if $\mathcal{F}$ is a thick filter on $[\omega]^m$, then $\lambda^{(n-m)} \mathcal{F}$ is a thick filter on $[\omega]^n$.
\vspace{3 mm}

\begin{defin}
	\label{MNSelective}
Let $p\in \beta [\omega]^m$ be a thick ultrafilter. We say that $p$ is \emph{$(m,n)$-selective} if $\lambda^{(n-m)} p$ is an ultrafilter. We say that $p$ is \emph{weakly $(m,n)$-selective} if there is a unique thick ultrafilter extending the filter $\lambda^{(n-m)} p$.
\end{defin}
\vspace{3 mm}

If $p\in \beta [\omega]^m$ is a thick ultrafilter and $q\in \beta [\omega]^n$ is a thick ultrafilter extending $\lambda^{(n-m)} p$, then we have $\pi^n_m(q) = p$. Therefore to prove Theorem \ref{TwoIdeals}, it is enough to construct a thick ultrafilter $p\in \beta [\omega]^2$ which is not weakly $(2,3)$-selective. Indeed, if $p\in\beta[\omega]^2$ is not weakly $(2,3)$-selective, then there are thick ultrafilters $q_0\neq q_1$ both extending the filter $\lambda p$, so $\pi^3_2(q_0) = \pi^3_2(q_1)$.

Our construction proceeds in two parts. First we define a certain type of pathological subset of $[\omega]^3$ and show that its existence allows us to construct $p\in \beta [\omega]^2$ which is not weakly $(2,3)$-selective. Then we show the existence of such a pathological set. 
\vspace{3 mm}

We begin by developing some abstract notions. Let $Y$ be a set, and let $\mathcal{I}$ be a proper ideal on $Y$. Write $S\subseteq_{\mathcal{I}} T$ if $S\setminus T\in \mathcal{I}$.  Let $\psi: \mathcal{P}(Y)\rightarrow \mathcal{P}(Y)$ be a map satisfying $\psi^2 = \psi$, $S\subseteq \psi(S)$, $S\subseteq T\Rightarrow \psi(S)\subseteq \psi(T)$, and $\psi(\emptyset) = \emptyset$. Call a set $S\subseteq Y$ \emph{$\psi$-closed} or just \emph{closed} if $\psi(S) = S$, and call $S$ \emph{near-closed} if there is a closed set $T$ with $S\Delta T\in \mathcal{I}$. Call a set $S\subseteq Y$ \emph{$(<\!\aleph_0)$-near-closed} if there are $k< \omega$ and closed $T_0,...,T_{k-1}$ with $S\Delta (\bigcup_{i<k} T_i)\in \mathcal{I}$. Notice that a finite union of near-closed sets is $(<\!\aleph_0)$-near-closed.

Now suppose $S\subseteq Y$ is a set which is not $(<\!\aleph_0)$-near-closed. If $p,q\in \beta Y$, we say that $p$ \emph{$\psi$-intertwines $q$ over $S$ modulo $\mathcal{I}$} if the following three items all hold:
\vspace{3 mm}

\begin{enumerate}
\item
$\{S\setminus T: T \text{ is near-closed and } T\subseteq_{\mathcal{I}} S\}\subseteq p$,
\vspace{2 mm}

\item
$\{\psi(T)\setminus S: T\in p, T\subseteq S\}\subseteq q$,
\vspace{2 mm}

\item
$p$ and $q$ extend the filterdual of $\mathcal{I}$.
\end{enumerate}
\vspace{3 mm}

If $\psi$, $S$, and $\mathcal{I}$ are understood, we will just say that $p$ \emph{intertwines} $q$. Notice in (1) that if $T$ is near-closed with $T\subseteq_{\mathcal{I}} S$, then $S\cap T$ is also near-closed, so it is enough to consider near-closed $T$ with $T\subseteq S$.
\vspace{3 mm}

\begin{lemma}
\label{IntertwineLemmas}
Fix $S\subseteq Y$ which is not $(<\!\aleph_0)$-near-closed.
\vspace{2 mm}

\begin{enumerate}
\item
If $B\subseteq Y$ with $B\in \mathcal{I}$, then $B$ is near-closed. Hence $S\not\in \mathcal{I}$.
\vspace{2 mm}

\item
There are $p, q\in \beta Y$ so that $p$ intertwines $q$.
\end{enumerate}
\end{lemma}
\vspace{1 mm}

\proof
The first part follows since the empty set is closed. 

Since $S$ is not $(<\!\aleph_0)$-near-closed, we have that $\{S\setminus T: T \text{ near-closed and } T\subseteq_{\mathcal{I}} S\}$ generates a filter $\mathcal{F}$ extending the filterdual of $\mathcal{I}$. Let $p\in \beta Y$ be any ultrafilter extending $\mathcal{F}$. 

Now let $T\in p$. Then $\psi(T)\setminus S\not\in \mathcal{I}$; otherwise we would have $\psi(T)\subseteq_{\mathcal{I}} S$, so $S\setminus \psi(T)\in p$, contradicting that $T\in p$. Also note by monotonicity of $\psi$ that $(\psi(T_0)\cap \psi(T_1))\setminus S\supseteq \psi(T_0\cap T_1)\setminus S$, so the collection $\{\psi(T)\setminus S: T\in p\}$ generates a filter $\mathcal{H}$ avoiding $\mathcal{I}$; letting $q$ be any ultrafilter extending both $\mathcal{H}$ and the filterdual of $\mathcal{I}$, we see that $p$ intertwines $q$.
\qedhere
\vspace{5 mm}

We now apply these ideas. Let $Y = [\omega]^3$, and let $\mathcal{I}$ be the thin ideal. Given $T\subseteq [\omega]^3$, view $T$ as a $3$-uniform hypergraph, and form the shadow graph $\partial T := \{\{a,b\}\in [\omega]^2: \exists c (\{a,b,c\}\in T)\}$. Define $\psi(T) = \lambda\partial T$. In words, $\psi(T)$ is the largest hypergraph with $\partial\psi(T) = \partial T$. More generally, we can set $Y = [\omega]^n$ and let $\mathcal{I}$ be the ideal of subsets of $[\omega]^n$ which are not thick. If $m< n$ and $T\subseteq [\omega]^n$, we set $\partial^{(n-m)}T = \{s\in [\omega]^m: \exists t\in [\omega]^{n-m} (s\cup t\in T)\}$. Then we can set $\psi(T) = \lambda^{(n-m)}\partial^{(n-m)}T$.
\vspace{3 mm}

\theorem
\label{IntertwineWorks}
Let $Y = [\omega]^n$, let $\mathcal{I}$ be the thin ideal, and let $\psi = \lambda^{(n-m)}\partial^{(n-m)}$ for some $m< n$. Suppose $S\subseteq [\omega]^n$ is not $(<\!\aleph_0)$-near-closed, and say $p,q\in \beta [\omega]^n$ where $p$ intertwines $q$. Then $\pi_m^n(p) = \pi_m^n(q)$.
\vspace{1 mm}

\proof
Suppose towards a contradiction that $p':= \pi_m^n(p) \neq q':= \pi_m^n(q)$ as witnessed by $A\subseteq [\omega]^m$ with $A\in p'$, $[\omega]^m\setminus A\in q'$. Then setting $B:= \{s\in [\omega]^n: [s]^m\subseteq A\}$ and $C:= \{s\in [\omega]^n: [s]^m\subseteq [\omega]^m\setminus A\}$, we have $B\in p$, $C\in q$, and $B\cap C = \emptyset$. Note that both $B$ and $C$ are $\psi$-closed. Since $p$ and $q$ are intertwined, we have $\psi(B\cap S)\setminus S\in q$, so in particular $B\setminus S\in q$. But since $C\in q$, this is a contradiction.
\qedhere
\vspace{6 mm} 

The next theorem along with Theorems \ref{IntertwineWorks} and \ref{ProductIdemp} will prove Theorem \ref{TwoIdeals}.
\vspace{2 mm}

\begin{theorem}
\label{ExistsNotClosed}
With $\mathcal{I}$ and $\psi$ as in Theorem \ref{IntertwineWorks}, there is $S\subseteq [\omega]^n$ which is not $(<\!\aleph_0)$-near-closed. 
\end{theorem}
\vspace{-1 mm}

\proof
The following elegant proof is due to Anton Bernshteyn. 

We take $S$ to be the random $n$-uniform hypergraph. Suppose towards a contradiction that $S$ was $k$-near-closed for some $k< \omega$. We write $S = S_0\cupdots S_{k-1}$ with each $S_i$ near-closed. Let $T_i\subseteq [\omega]^n$ be a $\psi$-closed set with $S_i\Delta T_i\in \mathcal{I}$, and write $T = \bigcup_{i < k} T_i$. So $S\Delta T\in \mathcal{I}$. This means that there is some $\ell< \omega$ so that the hypergraph $S\Delta T$ contains no clique of size $\ell$.

We now compute an upper bound on the number of induced subgraphs of $S$ that can appear on $N$ vertices $V:= \{v_0,...,v_{N-1}\}\subseteq \omega$. Since $S$ is the random $n$-uniform hypergraph, there must be  $2^{\binom{N}{n}}$ many possibilities. But by assumption, $S = T\Delta G$, where $G$ is some hypergraph with no cliques of size $\ell$. Since an induced subgraph of a $\psi$-closed graph is $\psi$-closed, each $T_i|_V$ is determined by $\partial^{(n-m)}(T_i|_V)$, so in particular, there are at most $2^{\binom{N}{m}}$ many possibilities for each $T_i|_V$, so at most $2^{k\binom{N}{m}}$ possibilities for $T|_V$. As for $G$, we need an estimate on the number of $\ell$-free $n$-uniform hypergraphs on $N$ vertices. It is a fact that for some constant $c> 0$ depending only on $\ell$ and $n$, we can find $c\binom{N}{n}$ subsets of $N$ of size $\ell$ which pairwise have intersection smaller than $n$. By a probabilistic argument, it follows that the proportion of $n$-uniform hypergraphs on $N$ vertices which are $\ell$-free is at most $$(1-2^{-\binom{\ell}{n}})^{c\binom{N}{n}} \leq 2^{-c\binom{\ell}{n}\binom{N}{n}}:= 2^{-d\binom{N}{n}}.$$
Multiplying together the number of choices for $T|_V$ with the number of choices for $G|_V$, we have that the number of possibilities for $S|_V$ is at most
$$(2^{(1-d)\binom{N}{n}})(2^{k\binom{N}{m}}) \ll 2^{\binom{N}{n}}.$$
This shows that $S$ is not $(<\aleph_0)$-near-closed.
\qedhere
\vspace{5 mm}

Let us now briefly discuss why Theorem \ref{TwoIdeals} implies Theorem \ref{PointsSmallestIdeal}. Recall (see [HS]) that in any compact left-topological semigroup $S$, the smallest ideal $K(S)$ is both the union of the minimal right ideals and the union of the minimal left ideals. The intersection of any minimal right ideal and any minimal left ideal is a group, so in particular contains exactly one idempotent. More concretely, if $M\subseteq S(G)$ is a minimal right ideal and $u\in M$ is idempotent, then $S(G)u$ is a minimal left ideal and $Mu = M\cap S(G)u$. All the groups formed in this way are algebraically isomorphic. When $S = S(G)$ for some topological group $G$, we can interpret this group as $\mathrm{aut}(M(G))$, the group of $G$-flow isomorphisms of $M(G)$. 

Fix $M\subseteq S(G)$ be a minimal subflow, and let $\phi: S(G)\rightarrow M$ be a $G$-map. Letting $p = \phi(1_G)$, then we must have $\phi = \lambda_p$. It follows that $\phi$ is a retraction iff $\phi = \lambda_u$ for some idempotent $u\in M$. Furthermore, if $p\in M$, then there is a unique idempotent $u\in M$ with $p = pu\in Mu$. It follows that for some $q\in M$ we have $\lambda_q\circ \lambda_p = \lambda_u$.

Now suppose $N\subseteq S(G)$ is another minimal right ideal, and that $x\neq y\in S(G)$ can be separated by a retraction $\psi$ onto $N$. Pick any $p\in M$ and form the $G$-map $\lambda_p\circ \psi$. Notice that $\lambda_p|_N$ is an isomorphism. For some $q\in M$ we have $\lambda_p\circ \psi = \lambda_q$. Then for some $r\in M$, we have $\lambda_r\circ \lambda_q = \lambda_u$ a retraction. It follows that $x$ and $y$ are also separated by $\lambda_u$. Hence the relation of being separated by a retraction does not depend on the choice of minimal subflow $M\subseteq S(G)$.

Now let $G = S_\infty$, and let $M\neq N$ be the minimal right ideals found in Theorem \ref{TwoIdeals}. Let $L$ be any minimal left ideal, and let $u\in M\cap L$ and $v\in N\cap L$ be idempotents. We will show that $u$ and $v$ cannot be separated by retractions, so let $\phi: S(G)\rightarrow M$ be a retraction. Then $\phi = \lambda_w$ for some idempotent $w\in M$. Then $\phi(u) = wu = u$ since idempotents in $M$ are left identities for $M$. But now consider $\phi(v) = wv$. By our assumption on $M$ and $N$, $wv$ is an idempotent. However, we must also have $wv\in M\cap L$ since $M$ and $L$ are respectively right and left ideals. It follows that $wv = u$, so $\phi(u) = \phi(v)$ as desired.
\vspace{5 mm}

\section{Proximal and Distal}

The technique of finding $M, N\subseteq S(G)$ minimal subflows with $J(M)\cup J(N)$ a semigroup allows for a quick solution to Ellis's problem for some Polish groups $G$. 

Recall from the introduction that a $G$-flow $X$ is called \emph{proximal} if every pair of points is proximal. Now suppose that $M(G)$ is proximal, i.e.\ that $M(G)$ is proximal. Then every element of $M$ is an idempotent; to see why, notice that it suffices to show that $M\cap L$ is a singleton whenever $L$ is a minimal left ideal. Indeed, suppose $u\neq p\in M\cap L$, with $u$ idempotent. Suppose that $(u, p)$ were proximal, i.e.\ that for some $q\in S(G)$ we have $uq = pq$. Since $M\cap L$ is a group with identity $u$, we must have $pu = p$. Now as $M$ is a minimal right ideal, find $r\in M$ with $uqr = u$. But then $pqr = puqr = pu = p$. This is a contradiction, so $(u,p)$ cannot be proximal.

A $G$-flow $X$ is \emph{distal} if every pair of non-equal points is distal. A useful fact is that $X$ is distal iff $E(X)$ is a group. If $M(G)$ is distal and $M\subseteq S(G)$ is a minimal subflow, then $J(M)$ is a singleton. To see this, note that if $u, v\in J(M)$, then $uv = vv = v$, so $(u,v)$ is a proximal pair. If $u\in J(M)$ is the unique idempotent, then the map $\phi: E(M)\rightarrow M$ given by $p\rightarrow u\cdot p$ is a $G$-flow isomorphism.

For automorphism groups $G$ with $M(G)$ proximal or distal, it follows that the conclusion of Theorem \ref{TwoIdeals} is automatic for \emph{any} two minimal right ideals $M\neq N$. The same argument for $S_\infty$ shows that any two idempotents of the same minimal left ideal cannot be separated by retractions. Of course, we need to know that $S(G)$ contains more than one minimal right ideal; see ([Ba], Corollary 11) for a proof of this fact. 

The following theorem collects some examples of Polish groups $G$ with $M(G)$ proximal.
\vspace{2 mm}

\begin{theorem}
\label{ProximalEllis}
Let $G$ be either $\homeo{2^\omega}$ or the automorphism group of the countably-infinite-dimensional vector space over a finite field. Then $S(G) \not\cong E(M(G))$
\end{theorem}
\vspace{3 mm}

The case where $M(G)$ is distal was considered in \cite{MNT}. They consider Polish groups $G$ with $M(G)$ metrizable which are \emph{strongly amenable}, meaning that there are no non-trivial proximal minimal flows. Clearly any group $G$ with $M(G)$ distal must also be strongly amenable. Using the main result from \cite{BYMT}, the relevant result from \cite{MNT} can be stated as follows.

\begin{theorem*}[\cite{MNT}, Theorem 4.3]
	Let $G$ be a Polish group with $M(G)$ metrizable, and suppose $G$ is strongly amenable. Then there is a short exact sequence $1\to H\to G\to K\to 1$ with $H$ extremely amenable and $K$ compact. Furthermore, $M(G)$ is the natural action of $G$ on $K$.
\end{theorem*}
\vspace{5 mm}

\section{Some ultrafilters on $[\omega]^2$}
This last section includes a short discussion of some ultrafilters motivated by the work in section \ref{SInfty}. The first main theorem of this section provides a counterpoint to Theorem \ref{TwoIdeals}. 

\begin{theorem}
	\label{ConsistentIdeal}
	It is consistent with ZFC that there is a minimal subflow $M\subseteq S(G)$ so that if $N\subseteq S(G)$ is a minimal subflow with $J(M)\cup J(N)$ a semigroup, then $M = N$.
\end{theorem}
\vspace{2 mm}

The second theorem points out a key difference between selective ultrafilters and $(2,3)$-selective ultrafilters (recall Definition \ref{MNSelective}). Recall that if $p, q\in \beta \omega$, then we say that $q\geq_{RK} p$ if there is a function $f: \omega\rightarrow \omega$ with $f(q) = p$. Another characterization of selective ultrafilters is that they are exactly the ultrafilters which are minimal in the Rudin-Keisler order (see \cite{Bo}). The next theorem shows that $(2,3)$-selectives can be very far from Rudin-Keisler minimal.
\vspace{2 mm}

\begin{theorem}
	\label{ConsistentRK}
	If $p\in \beta \omega$, there is a countably closed forcing extension $\mathbb{P}$ adding a $(2,3)$-selective ultrafilter $q$ with $q\geq_{RK} p$.
\end{theorem}
\vspace{2 mm}

As it turns out, these two theorems will both be proven using the same forcing construction. We define a forcing $\mathbb{P}$ which is very similar to a forcing defined by Laflamme [L]. A slightly more straightforward forcing would suffice for Theorem \ref{ConsistentIdeal} where we don't refer to a fixed $p\in \beta \omega$, but with a bit more work, we can prove both theorems.
\vspace{2 mm}

\begin{defin}
	Fix $p\in \beta \omega$. Write $\omega = \bigsqcup_n E_n$ with $|E_n| = n$. We define $\mathbb{P} = \langle P, \leq\rangle$ as follows.
	\vspace{1 mm}
	
	\begin{enumerate}
		\item 
		A condition $A\in P$ is a subset of $\omega$ so that for every $k< \omega$, we have $\{n< \omega: |A\cap E_n|\geq k\}\in p$.
		\vspace{2 mm}
		
		\item 
		We declare that $B\leq A$ iff $B\subseteq A$.
		
	\end{enumerate}
\end{defin}
\vspace{2 mm}

If $A, B\in \mathbb{P}$, we define $B\preceq A$ iff there is $k< \omega$ so that $\{m< \omega: |E_m\cap (B\setminus A)|\leq k\}\in p$. It is straightforward to see that $\langle P, \preceq\rangle$ is a separative pre-order which is equivalent to $\mathbb{P}$.
\vspace{2 mm}

\begin{lemma}
	$\mathbb{P}$ is countably closed.
\end{lemma}

\begin{proof}
	First notice that if $\langle A_n: n< \omega\rangle$ is a $\preceq$-decreasing sequence in $P$, then setting $A_n' = \bigcap_{i\leq n} A_i$, we have that $A_n'$ is $\preceq$-equivalent to $A_n$. So we may freely work with $\leq$-decreasing sequences.
	
	Suppose $\langle A_n: n<\omega\rangle$ is a $\leq$-decreasing sequence in $P$. Write $S(m,k) = \{n<\omega: |A_m\cap E_n|\geq k\}$. Note that $S(m,k)\in p$ for every $m, k < \omega$. Also, if $m\leq m'$ and $k\leq k'$, then $S(m',k')\subseteq S(m,k)$.
	
	For $m\geq 1$, we define $T_m = S(m,m)\setminus S(m+1, m+1)$. Note that if $m<\omega$, then $\bigcup_{n\geq m} T_m = S(m,m)$. If $m\geq 1$ and $n\in T_m$, then $|A_m\cap E_n|\geq m$. We form $B\in P$ by setting 
	$$B = \bigcup_{m\geq 1} \bigcup_{n\in T_m} A_m\cap E_n.$$
	For each $m\geq 1$, we have $\{n< \omega: |B\cap E_n|\geq m\} = S(m,m)\in p$, so $B\in p$. To see that $B\preceq A_m$, we note that $\{n<\omega: B\not\subseteq A_m\}\subseteq \omega\setminus S(m,m)$.
\end{proof}
\vspace{3 mm}

If $A\in \mathbb{P}$, we set $\tilde{A} = \bigcup_n [A\cap E_n]^2\subseteq [\omega]^2$. The next proposition will prove Theorem \ref{ConsistentRK}.

\begin{prop}
	Let $G\subseteq \mathbb{P}$ be generic. Then $\tilde{G}:= \{\tilde{A}: A\in G\}$ generates a thick ultrafilter on $[\omega]^2$ which is $(2,n)$-selective for every $n$. Furthermore, this ultrafilter is $RK$-above $p$.
\end{prop}

\begin{proof}
	Set $E_2 = \widetilde{1_{\mathbb{P}}}$, and suppose $E_2 = S\sqcup T$. Let $A\in p$. By Ramsey's theorem, there is some non-decreasing function $k\rightarrow b(2,k)$ increasing to infinity so that any $2$-coloring of the complete graph on $k$ vertices has a monochromatic clique of size $b(2,k)$. If $|A\cap E_N| = k$, then let $X_N\subseteq A\cap E_N$ be chosen so that $|X_N| = b(2,k)$ and $\tilde{X_N}\subseteq S$ or $\tilde{X_N}\subseteq T$. Define $S', T'\subseteq \omega$, placing $N\in S'$ or $N\in T'$ depending on which outcome happens. WLOG suppose $S'\in p$. Then letting $X = \bigcup_{N\in S'} X_N$, we have $X\in \mathbb{P}$, $X\leq A$, and $X$ decides whether $S$ or $T$ is in the filter generated by $\tilde{G}$.
	
	The argument that the ultrafilter generated by $\tilde{G}$ is $(2,n)$-selective is almost the exact same. By Ramsey's theorem, there is some non-decreasing function $k\rightarrow b(n,k)$ increasing to infinity so that any $2$-coloring of the complete $n$-uniform hypergraph on $k$-vertices has a monochromatic clique of size $b(n,k)$. Now letting $E_n = \lambda^{(n-2)}(E_2)$, fix a partition $E_n = S\sqcup T$. If $A\in \mathbb{P}$, we can in a similar fashion find $X\leq A$ deciding whether $S$ or $T$ is in the filter $\lambda^{(n-2)}(\tilde{G})$.
	
	Lastly, let $\psi: E_2\rightarrow \omega$ be so that $\psi(\{x,y\}) = n$ iff $\{x,y\}\subseteq E_n$. Then if $\mathcal{U}\in V[G]$ is the ultrafilter generated by $\tilde{G}$, then $\psi(\mathcal{U}) = p$.
\end{proof}

We now turn towards the proof of Theorem \ref{ConsistentIdeal}. To do this, we use Theorem \ref{ProductIdemp}. Working in $V[G]$, let $M_G\subseteq S(S_\infty)$ be the unique minimal subflow so that $\phi(\mathcal{F}_2^{M_G})$ is the ultrafilter generated by $\tilde{G}$. We need to show that $\{A(<_0, <_1): (<_0, <_1)\in A^{M_G}\}$ generates $\tilde{G}$. To see why this is, fix $A\in \mathbb{P}$. We may assume that if $A\cap E_n\neq\emptyset$, then $|A\cap E_n|\geq 2$. We will construct linear orders $<_0$ and $<_1$ so that $A(<_0, <_1) = \tilde{A}$. 

First write $\omega = \bigcup_n X_n$, where $X_0 = \omega\setminus A$ and $X_n = A\cap E_n$. Some of the $X_n$ may be empty, but this is fine. First define $<_0$ and $<_1$ on $X_0$ to be any linear orders which completely disagree. Suppose $<_0$ and $<_1$ have been defined on $X_0\cupdots X_{n-1}$. First define $<_0$ and $<_1$ on $X_n$ so that they agree. Now place $X_n$ $<_0$-below everything built so far and also $<_1$-above everything built so far. Then $A(<_0, <_1) = \tilde{A}$ as desired. This completes the proof of Theorem \ref{ConsistentIdeal}.
\vspace{5 mm}

The proof of Theorem \ref{ConsistentIdeal} suggests another type of ultrafilter on $[\omega]^2$ we can define. If $p\in \beta [\omega]^2$ is thick, define $A^p = \{(<_0, <_1): A(<_0, <_1)\in p\}$. As we saw in section \ref{SInfty}, $A^p$ is an equivalence relation on $\mathrm{LO}(\omega)$.

\begin{defin}
	\label{LOUlt}
	Let $p\in \beta [\omega]^2$ be a thick ultrafilter. We call $p$ a \emph{linear order ultrafilter} if $\{A(<_0, <_1): (<_0, <_1)\in A^p\}$ generates $p$. Call $p$ a \emph{weak linear order ultrafilter} if $p$ is the unique thick ultrafilter containing every $A(<_0, <_1)$ with $(<_0, <_1)\in p$.
\end{defin}

One can prove that there are thick ultrafilters $p\in \beta [\omega]^2$ which are not weak linear order ultrafilters, providing an alternate proof of Theorem \ref{TwoIdeals}. The proof is very similar to the proof that some $p\in \beta [\omega]^2$ is not weakly $(2,3)$-selective.
\vspace{3 mm}

We end with some open question about these ultrafilters.
\vspace{2 mm}

\begin{que}
	Does ZFC prove the existence of $(2,3)$-selective ultrafilters? Of linear order ultrafilters?
\end{que}
\vspace{1 mm}

\begin{que}
	Can there exist a weakly $(2,3)$-selective ultrafilter which is not $(2,3)$-selective? Same question for linear order ultrafilters.
\end{que}

The last question is motivated by Theorem \ref{ConsistentRK}. This shows that $(2,3)$-selective ultrafilters can exist arbitrarily high up in the Rudin-Keisler order.

\begin{que}
	Is it consistent with ZFC that the $(2,3)$-selective ultrafilters are upwards Rudin-Keisler cofinal?
\end{que}

\end{document}